\documentclass[letterpaper, 11pt]{article}
\usepackage{amsmath, amsthm, amssymb}
\usepackage{graphicx}
\usepackage{xcolor}
\usepackage{algorithm}
\usepackage{algpseudocode}
\usepackage[colorlinks,linkcolor=blue,citecolor=blue,urlcolor=magenta,linktocpage,plainpages=false]{hyperref}
\usepackage{caption}
\usepackage{subcaption}
\usepackage{setspace}
\usepackage[left=1in, right=1in, top=1in, bottom=1in]{geometry}
\usepackage{authblk}
\usepackage{changepage}

\newtheorem{lemma}{Lemma}
\newtheorem{theorem}{Theorem}
\newtheorem{corollary}{Corollary}
\newtheorem{proposition}{Proposition}
\newtheorem{fact}{Fact}
\newtheorem{remark}{Remark}

\newtheorem{definition}{Definition}
\newtheorem{observation}{Observation}

\newtheorem{assumption}{Assumption}

\newcommand{\OPT}{\textup{OPT}}
\newcommand{\OPTn}{\overline{\textup{OPT}}}
\newcommand{\ip}[2]{\langle #1, #2\rangle}

\newcommand{\I}{\mathcal{I}}

\title{A Theoretical and Computational Analysis  of Full Strong-Branching}
\author[1]{Santanu S. Dey\thanks{santanu.dey@isye.gatech.edu}}
\author[1]{Yatharth Dubey\thanks{yatharthdubey7@gatech.edu}}
\author[2]{Marco Molinaro\thanks{molinaro@inf.puc-rio.br}}
\author[1]{Prachi Shah\thanks{prachi.shah@gatech.edu}}
\affil[1]{School of Industrial and Systems Engineering, Georgia Institute of Technology}
\affil[2]{Computer Science Department, Pontifical Catholic University of Rio de Janeiro}
\date{\today}

\begin{document}
\maketitle
\begin{abstract}
Full strong-branching (henceforth referred to as strong-branching) is a well-known variable selection rule that is known experimentally to produce significantly smaller branch-and-bound trees in comparison to all other known variable selection rules. In this  paper, we attempt an analysis of the performance of the strong-branching rule both from a theoretical and a computational perspective. On the positive side for strong-branching we identify vertex cover as a class of instances where this rule provably works well. In particular, for vertex cover we present an upper bound on the size of the branch-and-bound tree using strong-branching as a function of the additive integrality gap, show how the Nemhauser-Trotter property of persistency which can be used as a pre-solve technique for vertex cover is being recursively and consistently used through-out the strong-branching based branch-and-bound tree, and finally provide an example of a vertex cover instance where not using strong-branching leads to a tree that has at least exponentially more nodes than the branch-and-bound tree based on strong-branching. On the negative side for strong-branching, we identify another class of instances where strong-branching based branch-and-bound tree has exponentially larger tree in comparison to another branch-and-bound tree for solving these instances. On the computational side, we conduct experiments on various types of instances like the lot-sizing problem and its variants, packing integer programs (IP), covering IPs,  chance constrained IPs, vertex cover, etc., to understand how much larger is the size of the strong-branching based branch-and-bound tree in comparison to the optimal branch-and-bound tree. The main take-away from these experiments is that for all these instances, the size of the strong-branching based branch-and-bound tree is within a factor of two of the size of the optimal branch-and-bound tree. 
\end{abstract}
%###########################################################
%###########################################################
%###########################################################

%\newpage

\section{Introduction} \label{sec:intro}
The branch-and-bound scheme, invented by Land and Doig~\cite{land1960automatic}, is the method of choice for solving mixed integer linear programs (MILP) by all modern state-of-the-art  MILP solvers. 

We present a quick outline of this well-known method for binary MILPs below; see ~\cite{wolsey1999integer,conforti2014integer} for more discussion on the branch-and-bound method. The optimal objective function value of the linear programming (LP) relaxation of a given MILP provides a dual bound (upper/lower bound for maximization-type/minimization-type objective respectively) on the optimal objective function value of the MILP. This LP relaxation corresponds to the root node of the branch-and-bound tree. In order to improve this bound and to find feasible solutions, after solving the LP corresponding to a node, the feasible region of the LP  is partitioned into two sub-problems which correspond to the child nodes of the given node. Such a partitioning of the feasible region for binary MILPs is usually accomplished in practice by selecting a variable $x_j$ that is currently fractional and adding the constraint $x_j =0$ to one of the child nodes and the constraint $x_j =1$ to the other child node. The process of partitioning the feasible regions of the LP at a node continues recursively for the child nodes (thus forming a tree) and is stopped (sometimes referred to as pruning a node), if one of the following conditions hold: (i) the LP at the node is infeasible, (ii) the LP's optimal solution is integer feasible, or (iii) the LP's optimal objective function value is worse than an already known integer feasible solution. The procedure terminates when all nodes have been pruned. 

Given an underlying LP solver, formally speaking, the \emph{branch-and-bound algorithm} is well-defined by fixing two rules:
\begin{itemize}
\item Rule for selecting an open node to be branched on next and,
\item Rule for deciding the variable to branch on.
\end{itemize} 
It is natural to measure the efficiency of a branch-and-bound algorithm by the number of nodes (corresponding to number of LPs solved) in the tree, i.e., lesser the number of nodes, faster the algorithm. 

\paragraph{Selecting an open node to branch on next (node selection rule).}
It is well established~\cite{wolsey1999integer} that the \emph{worst-bound rule} (i.e., select node with the maximum LP optimal objective function value for a maximization-type MILP or select node with minimum LP optimal objective function value for a minimization-type MILP) for selecting the next node to branch on, leads to small branch-and-bound trees. The intuition behind this is the following: one cannot ignore the node with worst-bound if one wants to solve the MILP. Thus, it is best to select to branch on it first.  In the rest of the paper, we always assume to use the worst-bound rule.

\paragraph{Deciding which variable to branch on (variable selection rule).} Given that the rule for selecting the node to be branched on is well-understood, much of the research in the area of branch-and-bound algorithms has focused on the topic of deciding the variable to branch on -- see for example~\cite{healy1964multiple,dakin1965tree,driebeek1966algorithm, benichou1971experiments,mitra1973investigation,forrest1974practical,breu1974branch,gauthier1977experiments,land1979computer,eckstein1994parallel,applegate1995finding,linderoth1999computational,achterberg2005branching}. Most of the above work develops various intricate greedy rules for determining the branching variable. 
A popular concept is that of \emph{pseudocost branching}: the value of pseudocost (variable with largest pseudocost gets branched on) keeps a history of the success (in terms of improving dual bound) of the variables on which branching has already been done. Many of the papers cited above differ in how pseudocost is initialized and updated during the course of the branch-and-bound tree. Other successful methods like \emph{hybrid branching} and \emph{reliability branching}~\cite{achterberg2005branching} are combinations of pseudocost branching and \emph{full strong-branching}, that we discuss next.

The focus of this work is \emph{full strong-branching}~\cite{applegate1995finding}, henceforth referred to as strong-branching for simplicity. This rule works as follows: branching on all the current fractional variables is computed (i.e., the child nodes are solved for every choice of fractional variable) and improvement measured in the left and right child node. Branching is now done on the variable with the most `combined improvement', where the combined improvement is computed as a `score' function of the left and right improvement.
Formally, let $z$ be the optimal objective function value of the LP at a given node, and let $z^{0}_j$ and $z^{1}_j$ be the optimal objective function values of the LPs corresponding to the child nodes where the variable $x_j$ is set to $0$ and $1$ respectively; we define $\Delta^{+}_j := z - z^1_j$ and $\Delta^{-}_j := z - z^0_j$ (assuming the MILP's objective function is of maximizing-type). Note that  $\Delta^+_j = +\infty$ is the child node with $x_j$ set to $0$ is infeasible. Similarly for $\Delta^-_j$. Two common  score functions used  are:
$$\textup{score}_{L}(j) = (1 - \mu)\cdot \textup{min} \{\Delta^{-}_j , \Delta^{+}_j \} + \mu \cdot \textup{max} \{\Delta^{-}_j , \Delta^{+}_j \}$$
for a constant $\mu \in [0,1]$, and
$$\text{score}_{P}(j) = \max(\Delta^+_j, \epsilon) \cdot \max(\Delta^-_j, \epsilon),$$
for a constant $\epsilon > 0$, where the first score is recommended in~\cite{linderoth1999computational,achterberg2005branching} (the paper~\cite{achterberg2005branching} recommends using $\mu  = 1/6$) and the second score function is recommended in~\cite{achterberg2007constraint}, where $\epsilon >0$ is chosen close to $0$ (for example, $\epsilon = 10^{-6}$) to break ties. We will refer to the first score function as the \textit{linear score function} and the second score function as the \textit{product score function}. Finally, the variable selected to branch on belongs to the set:
$$\arg \max_{j} \{\textup{score}(j)\}.$$
Empirically, strong-branching is well-known to produce significantly smaller branch-and-bound trees~\cite{achterberg2005branching} compared to all other known techniques, but is extremely expensive to implement as one has to solve $2K$ LPs where $K$ is the number of fractional variables for making just one branching decision. This experimentally observed fact is so well established in the literature that almost all recent methods to improve upon branching decisions are based on using \emph{machine learning techniques to mimic strong-branching} that avoid solving the $2K$ LPs, see for example~\cite{khalil2016learning,lodi2017learning,alvarez2017machine,balcan2018learning,gasse2019exact,gupta2020hybrid,nair2020solving}. Finally, see~\cite{le2017abstract} that describes a more sophisticated way to decide the branching variable based on left and right improvement rather than a static `score' function.

\subsection{Our contributions}

As explained in the previous section, empirically it is well understood that strong-branching produces very small trees in comparison to other rules. However, to the best of our knowledge there is \emph{no understanding of how good strong-branching is in absolute terms.} In particular, we would like to answer questions such as: 
\begin{itemize}
\item How large is the tree produced by strong-branching in comparison to the smallest possible branch-and-bound tree for a given instance? Answering this question may lead us to finding better rules. 
\item A more refined line of questioning: Intuitively, we do not expect strong-branching to work well for all types of MILP models and instances. On the other hand, it may be possible that for some classes of MILPs strong-branching based branch-and-bound tree may be quite close to the smallest possible branch-and-bound tree. It would be, therefore, very useful to understand the performance of strong-branching vis-\'a-vis different classes of instances.    
\end{itemize}
In this paper, we attempt an analysis of the performance of the strong-branching rule -- both from a theoretical and a computational perspective, keeping in mind the above questions.  
%Thus, the  main goal of this work is to analyze full strong-branching, both theoretically and computationally. In brief we obtain the following results:
\begin{itemize}
\item Strong branching is provably good: We show that for the vertex cover problem, the strong-branching rule has several benefits. First, we present a fixed parameter-type (FPT) result that uses the additive gap between the IP's and the LP's optimal objective function value to bound the size of the branch-and-bound tree using strong-branching. 

Nemhauser and Trotter~\cite{nemhauser1975vertex} proved that one may fix variables that are integral in the optimal solution of the LP relaxation of vertex cover and still find an optimal solution to the IP. Note that in the branch-and-bound tree, every node corresponds to a sub-graph of the original vertex-cover instance, and thus, ideally we would like to continue to use the Nemhauser-Trotter property at each node. Instead of designing a specialized implementation of branch-and-bound algorithm (where we fix variables that have an integer value in the LP optimal solution at each node), our second result is to show that strong-branching naturally incorporates ``fixing" the integral variables of the LP solution recursively and consistently thoughout the branch-and-bound tree. 

Finally, we construct an instance where strong-branching yields a branch-and-bound tree that is exponential-times smaller than a branch-and-bound tree generated using a very reasonable alternative variable selection rule. 
%One of the reasons the alternative branching rule performs badly is that it is unable to implement the Nemhauser-Trotter property consistently throughout the tree (See details in next section). 
\item Strong branching is provably bad: We present a class of instances where the size of the strong-branching based branch-and-bound tree is exponentially larger than a special branch-and-bound tree that solves these instances. In fact, the result we prove is stronger -- we show that if we only branch on variables that are fractional, then the size of the branch-and-bound tree is exponentially larger than the given special branch-and-bound tree to solve these instances. This special branch-and-bound tree branches on variables that are integral in the optimal LP solutions at certain nodes.
\item Computational evaluation of the size of strong-branching based branch-and-bound tree against the ``optimal" branch-and-bound tree\footnote{We write ``optimal" branch-and-bound tree with quotes, since the size of the optimal branch-and-bound tree depends not only on the branching decisions taken at each node, but also on the properties of the LP solver. We discuss this issue in detail in Section \ref{sec:main}.}: We first present a dynamic programming algorithm for generating the optimal branch-and-bound tree whose running time is $\text{poly}(\text{data}(\I)) \cdot 3^{O(n)}$ where $n$ is the number of binary variables. Then we conduct experiments on various types of instances like the lot-sizing problem and its variants, packing IPs, covering IPs,  chance constrainted IPs, vertex cover, etc., to understand how much larger is the size of the strong-branching based branch-and-bound tree  in comparison to the optimal branch-and-bound tree. The main take-away from these experiments is that for all these instances, the size of the strong-branching based branch-and-bound tree is within a factor of two of the size of the optimal branch-and-bound tree. 
\end{itemize}

To the best of our knowledge, this is the first such study of this kind on strong-branching, that provides a better understanding of why strong-branching often performs so well in practice, and gives insight into when an instance may be challenging for strong-branching. 

The rest of the paper is organized in the following fashion. In Section~\ref{sec:main} we formally present all our theoretical and computational results.  In Section~\ref{sec:vertexcoverproof} and Section~\ref{sec:bdg_ex}, we present proofs of the theoretical results presented in Section~\ref{sec:main}. In Section~\ref{sec:computing_opt} we present the details of the dynamic programming algorithm mentioned above. Finally, in Section~\ref{sec:experiments} we present all the details of our computational experiments.

\section{Main results}\label{sec:main}
\subsection{Theoretical results}\label{sec:theo}
Note that in this section (Section~\ref{sec:theo}), whenever we refer to strong-branching, we assume that it has been used in conjunction with the product score function~\cite{achterberg2007constraint} $\text{score}_{P}$, where $\epsilon = 0$ and we use the convention that $0\cdot \infty = 0$. Finally, we refer to the \textit{size} of a branch-and-bound tree to denote its number of nodes; we note that in any binary tree, the number of nodes is at most $2\ell + 1$, where $\ell$ is the number of leaves. 

\subsubsection{Strong branching works well for vertex cover}
There are simple MILPs that require exponential size branch-and-bound trees \cite{jeroslow1974trivial,chvatal1980hard,dey2021lower, jiang2021complexity,dadush2020complexity}. A common way to meaningfully analyze an algorithm with exponential worst-case performance is to show its performance with respect to some parameter \cite{roughgarden2019beyond, cygan2015parameterized}. Arguably the most well-studied problem in parameterized complexity is \textit{vertex cover}.
\begin{definition}[Vertex cover]\label{def:vc_ip}
The vertex cover problem over a graph $G = (V, E)$ can be expressed as the following integer program (IP)
\begin{equation*}
\begin{array}{ll@{}ll}
\text{min} & \displaystyle\sum\limits_{v \in V} x_v & \\
\text{s.t.}   &x_u + x_v \geq 1, & \quad uv \in E \\
              &x_v \in \{0,1\}, & \quad v \in V
\end{array}
\end{equation*}
Given an instance $\I$ of this IP, we let $L(\I)$ denote its LP relaxation (i.e. when the variable constraints instead are $x_v \in [0,1]$). We denote the optimal objective function value of an instance by $\OPT(\I)$ and the optimal objective function value of its LP relaxation by $\OPT(L(\I))$. We denote its additive integrality gap $\Gamma(\I) := \OPT(\I) - \OPT(L(\I))$. For results pertaining to vertex cover, we use $n$ to denote the number of vertices (i.e. $n := |V|$).
\end{definition}

\paragraph{Upper bound on the size of branch-and-bound tree using strong-branching.}
%To further study the effectiveness of branch-and-bound with strong-branching, we will try to understand its parameterized performance as an algorithm for Vertex Cover, arguably the most well-studied problem in parameterized complexity. Note that here we consider branch-and-bound trees using the worst bound node selection rule and the strong-branching variable selection rule with the \textit{product score function}, $\text{score}_{P}(j)$, as defined in Section \ref{sec:intro}.

Let $\mathcal{T}_S(\I)$ represent a branch-and-bound tree for solving instance $\I$ using strong-branching. Unfortunately, the number of nodes in this tree depends not only on the node selection rule and variable selection rule, but also on the underlying LP solver as well. For example, a solver may report an integral solution at a given node and allow us to prune the node. Another solver might report a different optimal solution to the LP, which is not integral. Therefore, we will be careful to not refer to \emph{the branch-and-bound tree generated by strong-branching}. Instead, we will use $\mathcal{T}_S(\I)$ to represent \emph{some branch-and-bound tree generated by strong-branching}. As such, we conduct a parameterized analysis of strong-branching as an algorithm for vertex cover. 

We show an upper bound on strong-branching for vertex cover parameterized by its additive integrality gap $\Gamma(\I)$.
\begin{theorem}\label{thm:vc_fpt}
Let $\I$ be any instance of vertex cover. Assume we break ties within the worst-bound rule for node selection rule by selecting a node with the largest depth.  Let $\mathcal{T}_S(\I)$ be {some branch-and-bound tree generated by strong-branching} with the above version of worst-bound node selection rule that solves $\I$. Then independent of the underlying LP solver used, 
$$|\mathcal{T}_S (\I)| \leq  2^{2\Gamma(\I) + 2} + \mathcal{O}(n).$$
\end{theorem}
Moreover, it is impossible to find another branch-and-bound rule that has a much better upper bound, so in this sense strong-branching is in the worst-case almost optimal for vertex cover parametrized by integrality gap. This is because of the following bad example.
%Moreover, this bound is tight within a factor of $2$ and additive loss of $\mathcal{O}(n)$. 
% 
\begin{remark}\label{rem:tight}
There is an instance $\I$ of vertex cover such that {any} branch-and-bound tree that solves $\I$ has size $2^{2\Gamma(\I) +1} -1$.  
\end{remark}
This is the instance of $m$ disjoint triangles presented in \cite{basu2020complexity}. Note that the smallest vertex cover in this instance has value $2m$ while the optimal solution to the LP relaxation has value $\frac{3}{2}m$, therefore $\Gamma(\I) = \frac{1}{2}m$. It follows from the discussion in \cite{basu2020complexity}, that all branch-and-bound trees for this instance have $2^m $ leaves, i.e., at least $2^{m +1} - 1$ nodes. 

We note that the result of Theorem~\ref{thm:vc_fpt} matches the guarantee of the classic parameterized-complexity algorithm that uses bounded search trees tailored to this problem; see Theorem 3.8 of \cite{cygan2015parameterized}.

We present the proof of Theorem~\ref{thm:vc_fpt} in Section \ref{sec:fpt}. 

\paragraph{Strong branching and persistency.}

Nemhauser and Trotter~\cite{nemhauser1975vertex} prove the following property regarding the LP relaxation of vertex cover. 
\begin{fact}[Persistency; Theorem 2 of \cite{nemhauser1975vertex}]\label{fact:persistency}
Let $\I$ be an instance of vertex cover, $\hat{x}$ be an optimal solution of $L(\I)$ and $I$ be the set of variables on which $\hat{x}$ is integer (i.e. $I = \{j : \hat{x}_j \in \{0,1\} \}$). Then, there exists an optimal solution $\hat{y}$ to $\I$ such that $\hat{y}$ agrees with $\hat{x}$ on all of its integer components (i.e. $\hat{y}_j = \hat{x}_j$ for all $j \in I$). 
\end{fact}

One way to use this property is to use it as a pre-solve routine, i.e., fix variables that are integral in the optimal solution of the LP relaxation, and then work with the sub-graph induced by the vertices with value $\frac{1}{2}$ in the optimal solution of the LP. (The extreme points of the LP relaxation of the vertex cover problem are half integral~\cite{nemhauser1974properties}.)
However, note that in the branch-and-bound tree, every node corresponds to a sub-graph of the original vertex-cover instance, and thus, ideally we would like to continue to use the Nemhauser-Trotter property at each node. Instead of designing a specialized implementation of branch-and-bound algorithm (where we fix variables that have an integer value in the LP optimal solution at each node), we show that strong-branching naturally incorporates ``fixing" the integral variables of the optimal solution of LP at each node recursively throughout the branch-and-bound tree. In fact, it does even better in the following sense: due to dual degeneracy, there may be  alternative linear programming optimal solutions with different corresponding sets of variables being integral. Strong branching ``avoids branching" on all the variables that are integral in any of the alternative optimal LP solutions, i.e., strong-branching is not fooled by the LP solver.  

In order to present our results, we need to define the notion of \emph{maximal set of integer variables} and present some properties regarding this set of variables. 

\begin{fact}[Lemma 1 in \cite{picard1977integer}]\label{fact:max_idx}
Consider instance of vertex cover and its LP relaxation. Let $x^1, x^2$ be two optimal solutions to this LP relaxation and let $I^1, I^2 \subseteq [n]$ be the indices of the integer valued variables in $x^1, x^2$ respectively. Then, there exists an optimal solution of the LP relaxation, $\hat{x}$,  such that the set of integer valued variables in $\hat{x}$ is $I = I^1 \cup I^2$.
\end{fact}

Based on the above fact we obtain the following observation: Given a vertex cover instance, all optimal solutions of the LP relaxation that have a maximal number of integral coordinates actually have the same set $I \subseteq [n]$ of integral coordinates. This, given an instance of vertex cover $\I$, we call this subset the \emph{maximal set of integer variables} and denote as $I(\I)$.  Fact~\ref{fact:persistency} then implies that there exists an optimal solution to the vertex-cover instance where the  \emph{maximal set of integer variables} are fixed to integer values from the corresponding values of an maximal optimal solution of the LP relaxation. 

As discussed before, in a branch-and-bound tree, each node corresponds to a vertex cover instance on a sub-graph. Therefore, we can define the notion of maximal set of integer variables at a given node $N$, which we denote as $I(\I, N)$. Given an instance of vertex cover $\I$, we refer to $\mathcal{TP}(\I)$ as a partial branch-and-bound tree for $\I$ if all the nodes of $\mathcal{TP}(\I)$ cannot be pruned. For such a partial branch-and-bound tree, we use $\mathcal{TP}^B(\I)$ to refer to the dual bound that one can infer from the partial tree. 

The strength of strong-branching with regards to the maximal set of integer variables at a node $N$ is explained by the next result: Let $j$ belong to the maximal set of integer variables at node $N$. If the LP solver returns an optimal solution with $x_j$ integral, then clearly we do not branch on this variable. However, if the LP solver returns an optimal solution where $x_j$ is fractional and we decide on branching on this variable  based on strong-branching, then it must be that we have ``nearly solved" the instance, i.e., the dual bound that can be inferred from the partial tree must be equal to the optimal objective function value of the instance.  Formally we have the following:

\begin{proposition}\label{prop:makekernelsense}
Let $\I$ be any instance of vertex cover. Assume we break ties within the worst-bound rule for node selection by selecting a node with the largest depth. Consider a partial tree $\mathcal{TP}_S$ generated by strong-branching with the above version of worst-bound node selection rule. Let $N$ be a node of this tree that is not pruned. If $j \in I(\I,N)$ and we decide to branch on variable $x_j$ at node $N$ (using strong-branching), then 
\begin{enumerate}
\item $\mathcal{TP}_S^B(\I) = \OPT(\I)$.
\item After branching on $x_j$,in at most $\mathcal{O}(n)$ further branchings the algorithm returns an integral optimal solution. 
\end{enumerate}
\end{proposition}

The above result shows that  we ``almost never" branch on maximal set of integer variables at a node if we use strong-branching. However, after branching on a variable that is not in the maximal set of integer variables, what happens to the set of maximal set of integer variables at the child nodes? A very favorable property would be if the  maximal set of integer variables of the parent node is inherited by the child nodes.  Otherwise, while we may not branch on $x_j$ where $j \in I(\I,N)$ at node $N$, but we may end up branching on $j$ for a child $N'$ of $N$ --- in other words, we are then not really ``fixing" variable $x_j$.  As it turns out the above bad scenario does not occur. Formally we have the following:

\begin{theorem}\label{thm:nt-monot}
Let $\I$ be an instance of vertex cover. Consider any internal node $N$ of $\mathcal{T}_S(\I)$ 
%\textcolor{red}{such that optimal LP objective function value at $N$ is less than $\OPT(\I)$}. 
and let $N'$ be a child node of $N$ that results from branching on $x_v$ where $v \not \in I(\I, N)$. Then, $I(\I, N) \subset I(\I, N')$.
\end{theorem}

%Thus, if $j$ is the maximal set of integer variables at node $N$, it continues to remain in the maximal set of integer variables for all the child nodes of $N$. Furthermore, we do not branch on such a varibale $x_j$ unless the instance is essentially solved (i.e., dual bound inferred from the tree 

Therefore, using Proposition~\ref{prop:makekernelsense} and Theorem~\ref{thm:nt-monot} together, we can conclude that when using strong-branching, we are essentially repeatedly using Nemhauser-Trotter property recursively and consistenly through-out the branch-and-bound tree: If $j$ is in the maximal set of integer variables at node $N$, it continues to remain in the maximal set of integer variables for all the child nodes of $N$; and we do not branch on such a variable $x_j$ at node $N$ or any of its children unless the instance is essentially solved (i.e., dual bound inferred from the tree equals the values of the IP). 
%We highlight again that strong-branching incorporates this persistency property for maximal set of integer variables at a given node $N$ and is not ``fooled" by alternative LP solutions into branching on a variable that could have been fixed at this node due to integrality in a different solution of the LP. 

Here we present an example to illustrate that there are variable selection rules for which the property described in Theorem \ref{thm:nt-monot} does not hold. See the instance $\I^*$ shown in Figure \ref{fig:monoton_ex}. Consider the following partial branch-and-bound tree, letting $N$ denote the root node. Observe that there is an optimal LP solution that sets $x_a = x_c = 1$ and $x_b = x_d = 0$; therefore, $I(\I^*, N) = \{a, b, c, d\}$. However, suppose that an adversarial LP solver returns the optimal basic feasible solution $x_a = x_b = x_c = x_d = \frac{1}{2}$. Suppose we branch on $x_b$ and consider the sub-problem resulting from $x_b = 1$, which we denote $N'$. The unique optimal solution to $N'$ sets $x_b = 1$ and $x_a = x_c = x_d = \frac{1}{2}$; therefore, $I(\I^*, N') = \{b\} \subset I(\I^*, N)$. In fact, this example can be easily extended to that of Theorem \ref{thm:sep_example} (see below) to show that this variable selection rule similarly results in an exponentially larger branch-and-bound tree as compared to one that is obtained by strong-branching. 

\begin{figure}[ht]
\centering
\includegraphics[width=0.4\textwidth]{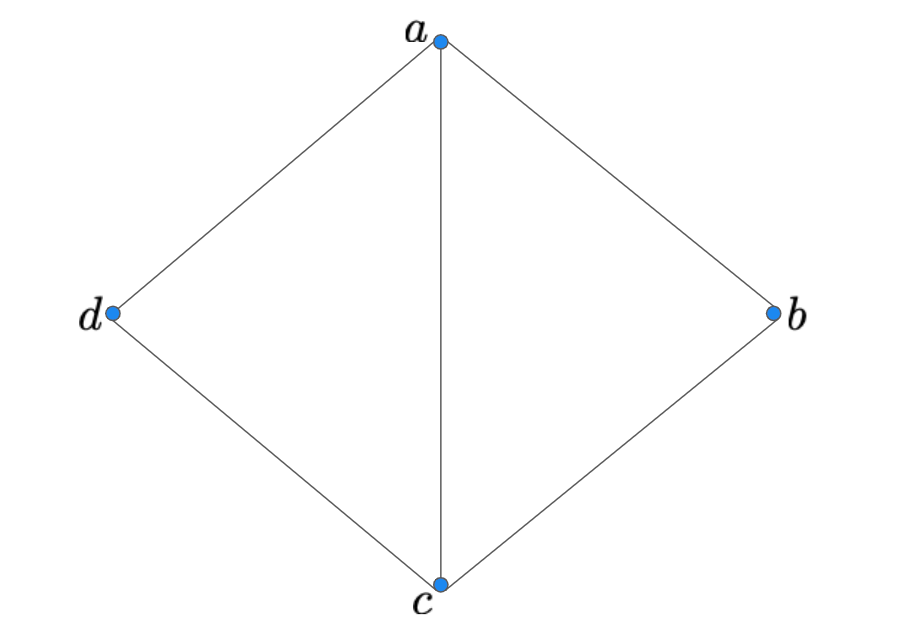}
\caption{Instance $\I^*$: example illustrating that the property of Theorem \ref{thm:nt-monot} is not true for every variable selection rule.}
\label{fig:monoton_ex}
\end{figure}

We present proof of Proposition~\ref{prop:makekernelsense} and Theorem~\ref{thm:nt-monot} in Section~\ref{sec:kernels}.

\paragraph{Superiority of strong-branching.}
There are very few papers that give upper bounds on sizes of branch-and-bound tree (when we use 0-1 branching)~\cite{dey2021branch,borst2021integrality}. These papers show that certain class of IPs with random data can be solved using polynomial-size branch-and-tree with high probability. However, these results do not depend on the variable selection rule used. Theorem~\ref{thm:vc_fpt} above is the first result of its kind that we are aware of, which uses a specific variable selection rule to prove upper bounds on size of branch-and-bound tree. To further highlight the importance of strong-branching in obtaining this upper bound result, we next show that if we do not use the strong-branching rule and we have an ``adversarial" LP solver, then we may need exponentially larger trees to solve the instance.

In particular, we next demonstrate the superiority of strong-branching by comparing it with another variable selection rule for the vertex cover problem (which we call the \emph{greedy-rule}) when the LP solver is adversarial, i.e. the LP solver always gives the most fractional extreme point solution. The greedy-rule for variable selection is intuitive and natural for vertex cover:  branch on the vertex with most fractional neighbors. Since setting such a vertex to $0$ would enforce all of its fractional neighbors to be $1$, one might think this rule provides the most ``local improvement''.

%Given an instance $\I$ of vertex cover, let $|\mathcal{T}_G (\I)|$ denote the number of leaves of the branch-and-bound tree generated by \textit{greedy} to solve the instance.
% 
\begin{theorem}\label{thm:sep_example}
Consider an LP solver with the following property: Among all optimal extreme point solutions, it reports an optimal extreme point with maximal number of fractional components. Given an instance $\I$ and the above LP solver, let $\mathcal{T}_S(\I)$ and $\mathcal{T}_G(\I)$ be some branch-and-bound trees that solves instance $\I$ obtained using the strong-branching rule and the greedy rule respectively. 

There is an instance $\I^*$ of vertex cover such that 
$$|\mathcal{T}_G(\I^*)| \geq 2^{\Omega(n)} \cdot |\mathcal{T}_S(\I^*)|$$ 
and furthermore 
$$|\mathcal{T}_G(\I^*)| \geq 2^{\text{cst}\cdot \Gamma(\I^*)}$$ 
where $\text{cst}$ is a constant strictly greater than $2$.
\end{theorem}

%As discussed in previous section, strong-branching is not fooled by the fractionality of a particular optimal solution given by the LP solver and continues to apply Nemhauser-Trotter Theorem recusively and consistently. In the above example, this can be viewed as the reason of better performance of strong branching rule.

We present the proof of Theorem~\ref{thm:sep_example} in Section~\ref{sec:sb_greedy_ex}.

\subsubsection{Strong branching does not work well for some instances}
Next we present a negative result regarding strong-branching, showing that strong-branching based branch-and-bound tree can have an exponential times as many nodes as compared to number of nodes in an alternative tree. In fact, the example shows something even stronger: any tree that branches only on variables fractional in the current nodes optimal solution will have exponential size, while an alternative tree has linear size. %Let $|\mathcal{T}^*(\I)|$ denote the number of leaves of the optimal branch-and-bound tree solving instance $\I$.

We begin by showing a seemingly surprising result about the existence of an extended formulation for any binary IP, that leads to a linear size branch-and-bound tree. 

\begin{proposition}\label{prop:bdg_ub}
For any integer program $\I$ with $n$ binary variables, there is an equivalent integer program that uses an extended formulation of the feasible region of $\I$ with $2n$ binary variables, which we refer to as $BDG(\I)$, that has the following property: there exists a branch and bound tree $\mathcal{T}^*(BDG(\I))$ that solves the instance $BDG(\I)$ and $|\mathcal{T}^*(BDG(\I))| \leq 4n + 1$.
%For any binary mixed-integer program $\I$, there is a pure binary integer program $BDG(\I)$ defined on $2n$ variables such that there exists a branch and bound tree $\mathcal{T}^*(BDG(\I))$ that solves the instance $BDG(\I)$ and $|\mathcal{T}^*(BDG(\I))| \leq 2n + 1$.
\end{proposition}

The extended formulation corresponding to $BDG(\I)$ used in Proposition~\ref{prop:bdg_ub} was first introduced in~\cite{bodur2017cutting} to show that every binary integer program has an extended formulation with split rank of $1$. We also remark here that Proposition~\ref{prop:bdg_ub} does not imply that the decision version of binary IPs is in \emph{co-NP}. This is because the $BDG(\I)$  formulation may be of exponential-size in comparison to the original formulation. 

In  Corollary~\ref{cor:bdg_lb} below, we take the cross-polytope~\cite{dey2021lower} and apply
 the extended formulation of Proposition~\ref{prop:bdg_ub} to obtain an example where strong-branching based branch-and-bound tree can have an exponential times as many nodes as compared to number of nodes in an alternative tree.
%\begin{corollary}\label{cor:bdg_lb}
%There exists an instance $\I$ such that the following holds: Let $\mathcal{T}$ be any tree that solve $BDG(\I)$ satisfying the following property: if $x$ is the optimal solution to an internal node $N$ of $\mathcal{T}$, then the variable $j$ branched on at $N$ must be such that $x_j \in (0,1)$. Then, $|\mathcal{T}| \geq 2^n$.   In particular, if $\mathcal{T}_S(BDG(\I))$ is a branch-and-bound tree generated using strong-branching that solve $BDG(\I)$, then $|\mathcal{T}_S(BDG(\I))| \geq 2^n$.  On the other hand,  $|\mathcal{T}^*(BDG(\I))| \leq 2n + 1$.
%\end{corollary}
\begin{corollary}\label{cor:bdg_lb}
There exists an instance $\I^*$ with $2n$ binary variables, such that the following holds: Let $\mathcal{T}(\I^*)$ be any tree that solves $\I^*$ satisfying the following property: if $x$ is the optimal solution to an internal node $N$ of $\mathcal{T}(\I^*)$, then the variable $j$ branched on at $N$ must be such that $x_j \in (0,1)$. Then, $|\mathcal{T}(\I^*)| \geq 2^{n +1} - 1$.   In particular, if $\mathcal{T}_S(\I^*)$ is a branch-and-bound tree generated using strong-branching that solve $\I^*$, then $|\mathcal{T}_S(\I^*)| \geq 2^{n +1} - 1$.  On the other hand,  there exists a tree $\mathcal{T}^*$ that solves $\I^*$ such that $|\mathcal{T}^*(\I^*)| \leq 4n + 1$.
\end{corollary}

%For any instance $\I$, the instance $BDG(\I)$ referred to in Proposition \ref{prop:bdg_ub} is a particular extended formulation given by Bodur, Dash and Gunluk~\cite{bodur2017cutting}. I

% apply Proposition~\ref{prop:bdg_ub}}. 

Typically, when implementing a branch-and-bound algorithm, one might be inclined to restrict the algorithm to branch on variables that are fractional in the current optimal solution. Therefore the above example is counter-intuitive in that it shows there can be a significant separation between branch-and-bound trees that are restricted to branch on variables that are fractional in the current optimal solution and branch-and-bound trees that are allowed to branch on integer valued variables. To our knowledge, this is the first such explicit example in the literature.

We present a proof of Proposition~\ref{prop:bdg_ub} and Corollary \ref{cor:bdg_lb} in  Section \ref{sec:bdg_ex}.
\subsection{Computational results}
In the previous section, we have shown that strong-branching works well for vertex cover and on the other hand, strong-branching can sometimes produce exponentially larger trees than alternative trees to solve an instance.  However, in general it seems very difficult to analyze strong-branching for general MILPs on a case-to-case basis. Moreover, we would really like to answer the question: how good is the strong-branching based branch-and-bound tree in comparison to the optimal tree? In this section we try to shed light on this question using computational experiments.

\paragraph{Optimal branch-and-bound tree.}
We begin with a discussion of an ``optimal branch-and-bound tree" for a given instance.% We will argue that it is not possible to have a well-defined optimal branch-and-bound tree unless we make some assumptions about the LP solver as well. 

First note that it is clear that the optimal branch-and bound tree uses the worst bound rule for node selection~\cite{wolsey1999integer}. Moreover, we will consider branching on all variables at a given node, whether it is integral or not in the optimal solution of the LP relaxation. Since we are using the worst bound rule for node selection, we only branch on nodes whose objective function value is at least as good as that of the MILP optimal objective function value.

Now consider a node whose LP optimal objective function value is equal to that of the MILP optimal objective function value. There are two possible scenarios:
\begin{itemize}
\item The optimal face of this LP is integral, i.e., all the vertices of the optimal face are integer feasible and therefore the LP solver (like simplex) is guaranteed to find an integral solution and prune the node.
%\item Every vertex of the optimal face is fractional. In this case, we are unable to prune this node. Eventually, depending on properties of other nodes and how we break ties for node selection among nodes with same objective function value, we may or may not end up branching on this node. 
%\item Some vertices are integral and some others are fractional. In this case, depending on whether the LP solver returns an integral vertex or not, we are able to prune the node or not. 
\item At least one vertex of the optimal face is fractional. In this case, depending on whether the LP solver returns an integral vertex or not, we are able to prune the node or not. Moreover, if we arrive at a fractional vertex, then depending on properties of other nodes and how we break ties for node selection among nodes with same objective function value, we may or may not end up branching on this node. 
\end{itemize}
In other words, in the second case, the size of tree may depend on the type of optimal vertex reported by the LP solver. In order to simplify our analysis and to remove all ambiguity regarding the definition of the optimal branch-and-bound tree, we make the following assumption for results presented in this section.

\begin{assumption}\label{assum:int}
We assume that if there exists an optimal solution to the LP relaxation at a given node that is integral,  then the LP solver finds it. 
\end{assumption}

\begin{observation}\label{obs:assumption}
When using an LP solver which satisfies Assumption \ref{assum:int}, a branch-and-bound tree using the worst bound rule  never branches on a node whose optimal objective function value is equal to that of the IP solver. 
\end{observation}
Proof of Observation~\ref{obs:assumption} is presented in Section~\ref{sec:opttreewelldefined}.
The above observation clearly makes the notion of ``optimal branch-and-bound tree" for a given instance well-defined.

\paragraph{Dynamic programming algorithm for finding the optimal branch-and-bound tree.}

In Section~\ref{sec:dpalgo} we present a dynamic programming (DP) algorithm that computes an optimal branch-and-bound tree for a given instance under Assumption~\ref{assum:int} for the LP solver. 
\begin{theorem}\label{thm:DP}
Under Assumption~\ref{assum:int}, there exists an algorithm with running time $\textup{poly}(data)\cdot3^{O(n)}$ time to compute an optimal branch-and-bound tree for any binary MILP instance $\I$ defined on $n$ binary variables. 
\end{theorem}

Section~\ref{sec:dpalgo} also presents some computational enhancements (like exploiting parallel computing) to improve the wall clock run time of the DP algorithm.

\paragraph{Computationally comparison of various variable selection rules against the optimal branch-and-bound tree.}

We conduct the first study comparing branch-and-bound trees employing different variable selection rules to the optimal branch-and-bound tree.  Detailed results are presented in Section~\ref{sec:experiments}.

%We use this algorithm to attain the computational results described below.

We consider the following variable selection rules: strong-branching with linear score function, strong-branching with product score function, most infeasible, and random. We evaluate the performance of these rules on a wide range of problems: general packing and covering IPs (packing-type, covering-type, and mixed packing and cover instances), lot-sizing and variants, vertex cover, chance constraint programming (CCP) models for multi-period power planning and portfolio optimization, and stable set on a bipartite graph with knapsack sides constraint (100 instances for each model). Note that all of these instances have up to $20$ binary variables since we are unable to run the DP algorithm of Theorem~\ref{thm:DP} for larger instances. See Section~\ref{sec: Data generation} for a detailed description of all these instances. 

We present a discussion of all our results (together with tables and explanatory figures) in Section~\ref{sec:compres}. Here are some of our notable findings on sizes of branch-and-bound tree trees:
\begin{itemize}
    \item Random consistently performs the worst.
    \item Strong branching \textit always performs the best. 
    \item While the performance of two variants of strong-branching is comparable on all problems considered in this study, strong branching with product score function (SB-P) dominates over strong branching with linear score function (SB-L) on 8 out of 10 problems, although by a small margin.
    \item The geometric mean of branch-and-bound tree size using strong-branching remains less than twice the size of optimal branch-and-bound tree for all problems considered in this study. This is not that case for all branching rules: the Random rule (and sometimes even the most-infeasible rule) is typically many more times larger than the optimal branch-and-bound tree.
\end{itemize}
Finally, see Figure \ref{fig:gmratio} for a summary of the computational results. It is clear that while strong-branching does quite well (within a factor of $2$), there is clearly scope for coming up with better rules for deciding branching variables, for example for problems such as CCP portfolio optimization. It would be interesting to see if one can use machine learning techniques to learn from the optimal branch-and-bound trees.

\begin{figure}[ht]
    \centering
    \includegraphics[width=\linewidth]{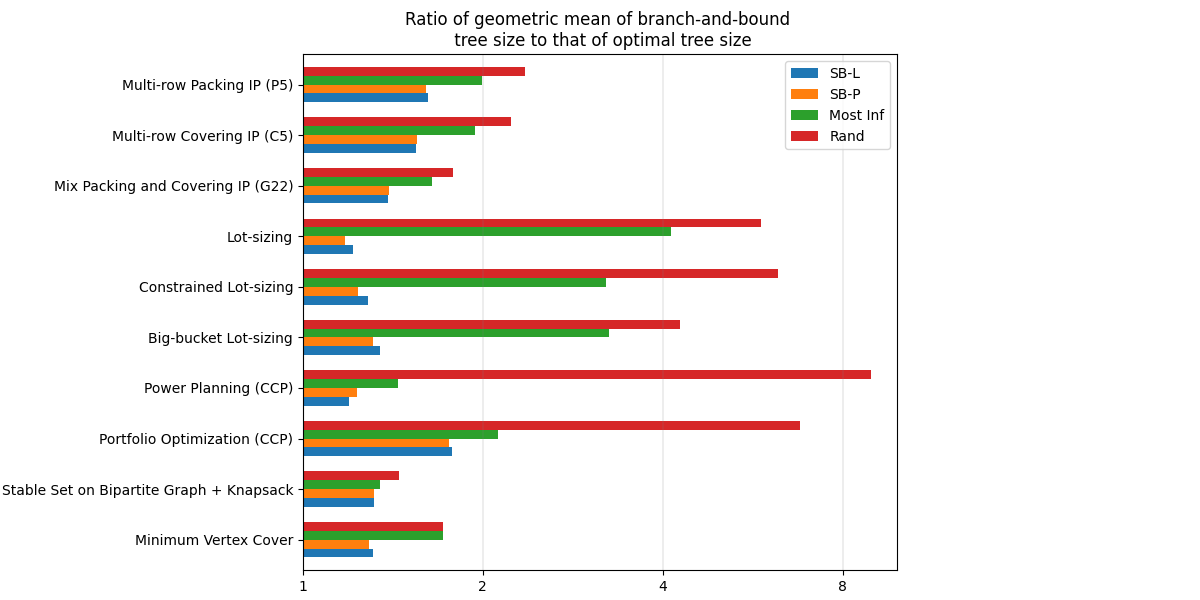}
    \caption{Ratio of geometric mean of branch-and-bound tree sizes to geometric mean of optimal tree sizes over all instances of a problem for various branching strategies. ``Rand" stands for random, ``Most Inf" stands for most infeasible, ``SB-P" stands for strong-branching with product score function, and ``SB-L" stands for strong-branching with linear score function.}
    \label{fig:gmratio}
\end{figure}

In previous section (see Corollary~\ref{cor:bdg_lb}), we have seen an example where strong-branching performs badly while alternative branch-and-bound tree that has exponentially lesser nodes, branches on integer variables. So a natural question we would like to understand is the percentage of times the optimal branch-and-bound tree branches on integer variables for the various instances. These results are presented in Table~\ref{tab:ratio}.  We note that there are multiple optimal trees. So it may be possible that these exists other optimal trees with a slightly different number of branchings on integer variables.  

Here are some of our notable findings:
\begin{itemize}
\item Strong branching does relatively poorly on general packing and covering IPs which has a high fraction of integer branchings in the optimal tree and it does very well on lot-sizing where the optimal tree rarely branches on integers. So it would appear consistent with our hypothesis that more branchings on integer variables in the optimal tree implies strong-branching performs poorly.
%\item However, in the case of CCP portfolio optimization problem, strong-branching has a relatively poor performance, yet the optimal tree almost always branches on fractional values. 
\item On the other hand, for stable set on bipartite graph with knapsack side constraint, strong-branching does very well in spite of a lot of integer branchings in the optimal tree. 
\end{itemize}

\begin{table}[ht]

\begin{center}\small

\caption{Summary of average tree sizes (geometric) and percentage of branching on integral variable in optimal tree for all problems across 100 instances}\label{tab:ratio}

% \begin{adjustwidth}{-0.5in}{-0.5in}

\begin{tabular}{lrrrrrr}

% & & & & & & Fraction of \\

% & & & & & & integer branchings\\

Problem                                    & Opt Tree & SB-L   & SB-P   & Most inf & Rand   & \% Int Branch \\
        	
\cline{1-7}

Multi-row Packing IP (P5)&25.4 & 41.2 & 40.9 &50.8 & 59.9 &48.2\% \\
Multi-row Covering IP (C5)&25.2 & 39.0 &39.1&48.9&56.4&49.0\% \\
Mix Packing and Covering IP (G22) & 10.2 & 14.2 & 14.2 & 16.8 & 18.3 & 48.0\% \\
Lot-sizing & 111.5 & 135.0 & 131.1 & 461.0 & 651.6 & 0.2\% \\
Constrained Lot-sizing & 101.9 & 131.1 & 125.9 & 328.0 & 635.0 & 0.2\% \\
Big-bucket Lot-sizing & 81.8 & 110.3 & 107.1 & 266.4 & 349.6 & 1.1\% \\
Power Planning (CCP) & 37.9 & 45.3 & 46.8 & 54.8 & 337.9 &1.1\% \\ 
Portfolio Optimization (CCP)&97.4 & 172.8 & 171.1 & 206.5 & 659.1 & 4.4\% \\
Stable Set on Bipartite Graph + Knapsack & 137.6 & 180.8 &180.8 & 185.5 & 199.6 & 47.7\% \\
Minimum Vertex Cover & 7.1 & 9.3 & 9.1 & 12.1 & 12.2 & 0.0\% \\
\cline{1-7}
\end{tabular}

% \end{adjustwidth}

\end{center}
\end{table}

%\newpage

In other words, there does not seem to be a direct relationship between the performance of strong-branching and the number of branching on integer variable of the optimal branch-and-bound tree. 

Finally we end this section with a word of caution regarding over-interpreting the computational results above: As mentioned above, due of the exponential nature of the DP algorithm to compute the optimal branch-and-bound tree, computational experiments could be performed only on relatively smaller problem sizes with up to 20 binary variables. Some of the observations derived here may not extrapolate to larger instances. 

\section{Analysis of Strong Branching for Vertex Cover}\label{sec:vertexcoverproof}

%\subsection{Strong Branching is FPT}\label{sec:fpt}
\subsection{Proof of Theorem~\ref{thm:vc_fpt}.}\label{sec:fpt}
% \begin{assumption}\label{ass:lp}
% We have an LP solver that finds an optimal solution such that the set of integer variables is maximal.
% \end{assumption}

% Now we formally state the result and present its proof.

% \red{For this result, we assume that if there is an optimal solution that is integral, the LP solver returns it.} We will require the following property from \cite{nemhauser1974properties}.
 Throughout this section, we use $N_{j, 0}$ to denote the child node of $N$ that results from the branch $x_j = 0$; we use $N_{j, 1}$ similarly.

Throughout this section we make the following assumptions. Let $\I$ be any instance of vertex cover. Assume we break ties within the worst-bound rule for node selection rule by selecting a node with the largest depth.  Let $\mathcal{T}_S(\I)$ be \emph{some branch-and-bound tree generated by strong-branching} with the above version of worst-bound node selection rule that solver $\I$. Moreover, all results are independent of the underlying LP solver used.

We will require two preliminary results for the proof of Theorem~\ref{thm:vc_fpt}.

\begin{lemma}\label{lem:branch_frac}
Let $N$ be a node of $\mathcal{T}_S(\I)$ with optimal objective value  less than $\OPT(\I)$. Then strong-branching will branch on $v \not \in I(\I, N)$ at node $N$ and $N_{v, 0}, N_{v, 1}$ have optimal value at least $\frac{1}{2}$ more than that of $N$. 
\end{lemma}

\begin{proof}
Since $N$ has optimal objective value less than $\OPT(\I)$, the LP relaxation at $N$ must not have an optimal solution that is integer. Note that if at $N$ we branch on $x_u$ where $u \in I(\I, N)$, it must hold that, at least one of $N_{u, 0}$ or $N_{u, 1}$ have optimal value the same as $N$. Therefore, $\text{score}_{P}(u) = 0$. We now show that there exists $v \in [n]$ such that $\text{score}_{P}(v) > 0$. If there is no optimal solution to $N$ that is integer, then $I(\I, N) \not = [n]$ by Fact \ref{fact:max_idx}. Therefore, there exists a $v \not \in I(\I, N)$ such that $x_v = \frac{1}{2}$ in every optimal solution to $N$. It follows that fixing $x_v \in \{0,1\}$ must result in a feasible solution to $N$ that has strictly greater value, and so branching on $x_v$ at $N$ leads to child nodes $N_{v, 0}$ and $N_{v, 1}$ where both have optimal value more than that of $N$. Further, since $N_{v, 0}$ and $N_{v, 1}$ have objective value strictly more than $N$ and all basic feasible solutions are half-integral, it holds that $N_{v,0}$ and $N_{v, 1}$ have objective value at least $\frac{1}{2}$ more than $N$. 
\end{proof}

\begin{lemma}\label{lem:branch_int}
Let $N$ be a node of $\mathcal{T}_S(\I)$ with an optimal LP solution that is integer. Then the sub-tree rooted at $N$ will have size $\mathcal{O}(n)$, where it finds an integer optimal solution. 
\end{lemma}

\begin{proof}
Let $y \in \{0,1\}^n$ be an optimal solution to the LP of node $N$. If the LP solver returns an integer solution, we are done. Suppose not, and suppose we branch on some variable $v$, and without loss of generality let $y_v = 0$. Then, of course $N_{v, 0}$ will have $y$ as an optimal solution and therefore have value $\OPT(\I)$. If $N_{v, 1}$ has value greater than $ \OPT(\I)$, this node gets pruned by bound and we continue by branching on $N_{v, 0}$ since it is now the open node with largest depth. Suppose instead $N_{v, 1}$ has value $\OPT(\I)$. Then, we argue in the next paragraph that there is an optimal solution to the LP corresponding to $N$ that is integer feasible, denote $z$, with $z_v = 1$; in this case we can break the tie between $N_{v, 0}, N_{v, 1}$ arbitrarily. Then, for every node $N'$ in the sub-tree rooted at $N$, either one child of $N'$ has value greater than $\OPT(\I)$ and is pruned by bound, or both children of $N'$ have optimal solutions that are integer. In the second case, only one of these two children will ever be branched on, since we break ties by branching on the node with largest depth, and therefore will find an integer solution before revisiting a shallower node. The result of the lemma follows.

Here we argue that if $N_{v,1}$ has a LP optimal objective function value $\OPT(\I)$, then there is an optimal solution to the LP corresponding to $N$ that is integral, denote $z$, with $z_v = 1$. Let $x'$ denote any optimal solution of $N_{v,1}$ and observe that $x',y$ are both optimal solutions to $N$. It follows from the proof of Lemma 1 in \cite{picard1977integer} that: given two half-integral optimal solutions $x', y$, there exists an optimal solution $z$ constructed as follows:
$$ z_j = \begin{cases}
1 & \text{ if } x'_j = 1 \text{ or } x'_j = \frac{1}{2}, y_j = 1\\
0 & \text{ if } x'_j = 0 \text{ or } x'_j = \frac{1}{2}, y_j = 0 \\
\frac{1}{2} & \text{ if } x'_j = y_j = \frac{1}{2}
\end{cases}
$$
In particular,~\cite{picard1977integer} shows that that $z$ constructed as above is feasible and satisfies, $\sum_{j}z_j = \sum_{j}x'_j = \sum_{j}y_j$. Clearly $z_v = x'_v = 1$ and since $y \in \{0,1\}^n$, it follows that $z \in \{0,1\}^n$ and is an optimal solution of the LP corresponding to $N$. 
\end{proof}

\begingroup
\def\thetheorem{\ref{thm:vc_fpt}}
\begin{theorem}
Let $\I$ be any instance of vertex cover. Assume we break ties within the worst-bound rule for node selection rule by selecting a node with the largest depth.  Let $\mathcal{T}_S(\I)$ be \emph{some branch-and-bound tree generated by strong-branching} with the above version of worst-bound node selection rule that solver $\I$. Then independent of the underlying LP solver used, 
$$|\mathcal{T}_S (\I)| \leq 2^{2\Gamma(\I) + 2} + \mathcal{O}(n).$$
\end{theorem}
\addtocounter{theorem}{-1}
\endgroup

\begin{proof}
Fix any branch-and-bound tree (using strong-branching) for $\I$. Let $N$ be a node at depth $2\Gamma(\I) + 1$. Suppose that no ancestor of $N$ had an optimal LP solution that was integer; it follows from Lemma \ref{lem:branch_frac} that $N$ has optimal objective value at least $\OPT(\I) + \frac{1}{2}$. Denote the set of such nodes $\mathcal{N}_{\text{no ancestor}}$ and note that these nodes are pruned by bound.  Observe all other nodes (if any) at depth $2\Gamma(\I) + 1$ do have such an ancestor. Let $\mathcal{N}_{\text{integer}}$ denote the set of nodes at depth $\leq 2\Gamma(\I)$ that have an optimal solution that is integer, but none of its ancestors have an optimal solution that is integer. Finally, observe that $|\mathcal{N}_{\text{no ancestor}}| + |\mathcal{N}_{\text{integer}}| \leq 2^{2\Gamma(\I) + 1}$. We conclude the proof by observing that Lemma \ref{lem:branch_int} gives: if $N$ is the node in $\mathcal{N}_{\text{integer}}$ with largest depth, then the sub-tree rooted at $N$ has size $\mathcal{O}(n)$ where it finds an integer solution with value $\OPT(\I)$. Therefore, no other nodes in $\mathcal{N}_{\text{integer}}$ will be branched on. Since the number of leaves in this tree is at most $2^{2\Gamma(\I) + 1} + \mathcal{O}(n)$, the result of the theorem follows. 

\end{proof}

\subsection{Proof of Proposition~\ref{prop:makekernelsense} and Theorem~\ref{thm:nt-monot}.}\label{sec:kernels}

\begingroup
\def\theproposition{\ref{prop:makekernelsense}}
\begin{proposition}
Let $\I$ be any instance of vertex cover. Assume we break ties within the worst-bound rule for node selection by selecting a node with the largest depth. Consider a partial tree $\mathcal{TP}_S$ generated by strong-branching with the above version of worst-bound node selection rule. Let $N$ be a node of this tree that is not pruned. If $j \in I(\I,N)$ and we decide to branch on variable $x_j$ at node $N$ (using strong-branching), then 
\begin{enumerate}
\item $\mathcal{TP}_S^B(\I) = \OPT(\I)$.
\item After branching on $x_j$, in at most $\mathcal{O}(n)$ further branchings the algorithm returns an integral optimal solution.
\end{enumerate}
\end{proposition}
\addtocounter{proposition}{-1}
\endgroup

\begin{proof}
We begin by proving property 1. Suppose for sake of contradiction, there exists an open node $N'$ in $\mathcal{TP}_S^B(\I)$ with optimal objective value less than $ \OPT(\I)$. It follows from the worst-bound rule and Lemma \ref{lem:branch_frac} that strong-branching will choose to branch on $j'$ at $N'$ with $j' \not\in I(\I, N')$. Thus $j'$ and $N'$ are not the same as $j$ and $N$, giving the desired contradiction.

Property 2 follows directly from Lemma \ref{lem:branch_int} and the fact that $I(\I,N) = [n]$, since otherwise strong-branching would branch on some $j \not \in I(\I,N)$. 
\end{proof}

\begingroup
\def\thetheorem{\ref{thm:nt-monot}}
\begin{theorem}
Let $\I$ be an instance of vertex cover. Consider any internal node $N$ of $\mathcal{T}_S(\I)$  and let $N'$ be a child node of $N$ that results from branching on $x_v$ where $v \not \in I(\I, N)$. Then, $I(\I, N) \subset I(\I, N')$.
\end{theorem}
\addtocounter{theorem}{-1}
\endgroup

\begin{proof}
Throughout this proof, we use $N_{j, 0}$ to denote the child node of $N$ that results from the branch $x_j = 0$; we use $N_{j, 1}$ similarly. We use $\delta(S)$ to denote the neighbors of any subset of vertices $S \subseteq V$. Also, throughout the proof, let $x^*$ be an optimal solution to $N$ with maximal set of integer components (i.e. $x^*_j \in \{0,1\}$ for all $j \in I(\I, N)$).

We will first consider the child $N_{v, 1}$. Consider the solution $y$ to $N_{v, 1}$ constructed from $x^*$ as follows: $y$ takes the same values as $x^*$, but $y_v = 1$ instead of $\frac{1}{2}$. Note that this is feasible for $N_{v, 1}$. Since $v \not \in I(\I, N)$ we have that $x_v = \frac{1}{2}$ in every optimal solution to $N$. In other words, optimal objective function value of $N_{v,1}$ is at least $\frac{1}{2}$ more than that of $N$. Also, $\ip{1}{y} = \ip{1}{x^*} + \frac{1}{2}$, and so, $y$ is an optimal solution to $N_{v, 1}$. So $I(\I, N) \subset I(\I, N')$ follows in the case of $N' = N_{v, 1}$. 

We will now consider the child $N_{v, 0}$. Let $V_1^* = \{u \in V : x^*_u = 1\}$ and $V_0^*$ be defined similarly. We will need the following technical result later in the proof.

\begin{lemma}\label{lem:neighbors}
$|\delta(S) \cap V_0^*| \geq |S|$ for all $S \subseteq V_1^*$. 
\end{lemma}

\begin{proof}
Assume, for the sake of contradiction, there is a subset $S \subseteq V_1^*$ such that $|\delta(S) \cap V_0^*| < |S|$, then it must hold that $x^*$ is not an optimal solution. This is because, we can set all of $S$ and $\delta(S) \cap V_0^*$ to be $\frac{1}{2}$. This remains feasible, since we are only decreasing the value of the vertices in $S$, but all of the vertices with value $0$ adjacent to these vertices are also being raised to $\frac{1}{2}$. It also decreases the cost, since $\frac{1}{2} (|\delta(S) \cap V_0^*| + |S|) < |S|$. 
\end{proof}

Let $x'$ be any optimal solution to LP corresponding to $N_{v, 0}$. Let $\Phi = \{u \in V_1^* : x'_u \not = 1 \}$. We will construct a feasible solution $z$ such that $z_u = 1$ for all $u \in V_1^*$ with cost no more than
%at most 
that of $x'$. Note that since $\delta(V_0^*) \subseteq V_1^*$, this implies the existence of an optimal solution to $N_{v, 0}$ with the same integer variables as $x^*$ (since any optimal solution with all variables in $V_1^*$ set to $1$ will have all variables in $V_0^*$ set to $0$). 

We begin by constructing an intermediate solution $y$. Notice $\Phi$ can be partitioned into $\Phi^0 = \{u \in \Phi : x'_u = 0\}$ and $\Phi^{\frac{1}{2}} = \{u \in \Phi : x'_u = \frac{1}{2}\}$. Let $y$ be defined by setting the variables in $\Phi^0$ to $\frac{1}{2}$ and setting the variables in $\delta(\Phi^0) \cap V_0^*$ to $\frac{1}{2}$ (note that $x'_u = 1$ for all $u \in \delta(\Phi^0) \cap V_0^*$). This maintains feasibility, since the variables with decreasing value are a subset of $V_0^*$, which are only adjacent to vertices in $V_1^*$ which now all have value at least $\frac{1}{2}$ in $y$. This also maintains optimality, since the cost increases by $\frac{1}{2}|\Phi^0|$ and decreases by $\frac{1}{2}|\delta(\Phi^0) \cap V_0^*|$, and by Lemma \ref{lem:neighbors}, $|\delta(\Phi^0) \cap V_0^*| \geq |\Phi^0|$. We now construct $z$ similarly, from this intermediate solution $y$. Notice now, that all variables in $\Phi$ have value $\frac{1}{2}$ in $y$. We construct $z$ by setting all of these variables to $1$ and setting all variables in $\delta(\Phi) \cap V_0^*$ to $0$ (note that $y_u \geq \frac{1}{2}$ for all $u \in \delta(\Phi) \cap V_0^*$). This maintains feasibility, since the variables with decreasing value are a subset of $V_0^*$, which are adjacent only to vertices in $V_1^*$, which all have value at least $1$ in $z$. This also maintains optimality, since the cost increases by $\frac{1}{2}|\Phi|$ and decreases by at least $\frac{1}{2}|\delta(\Phi) \cap V_0^*|$, and by Lemma \ref{lem:neighbors}, $|\delta(\Phi) \cap V_0^*| \geq |\Phi|$. Finally, we note that $z$ maintains feasibility for $N_{v, 0}$, since $x'_v = 0$ and therefore $x'_u = 1$ for all $u \in \delta(v)$; since $\delta(v) \cap \Phi =  \emptyset$, all vertices $v$ and $\delta(v)$ keep their value in $z$. This concludes the proof. 

\end{proof}

%Therefore, the Buss and Nemhauser-Trotter kernels are implicitly used by strong-branching. 

\subsection{Proof of Theorem \ref{thm:sep_example}.}\label{sec:sb_greedy_ex}
%\subsection{Proof of Theorem \ref{thm:sep_example}: When strong-branching can be much better}\label{sec:sb_greedy_ex}

In this section, we will demonstrate an instance of vertex cover where the worst tree can have $2^{\Omega(n)}$ times as many nodes as the tree generated by strong-branching. Here we assume that the LP solver finds a basic feasible solution at every node, but the one with the \textit{most} fractional components.
% \red{(Should we motivate this? There is a fast algorithm to find the half-integral solution with the \textit{fewest} fractional components, but it is a combinatorial algorithm. Not sure what the LP solvers do that are used in IP solvers.)} 

\paragraph{The instance.} Here we refer to a \textit{triangle} as a $K_3$ and to a \textit{diamond} as a $K_4$ minus an edge; see Figure \ref{fig:graphs}. Let $G$ be a union of $m$ triangles and $\frac{1}{2}m$ diamonds. We will consider the instance of finding the maximum unweighted vertex cover on $G$. Observe the following properties of this instance:
\begin{itemize}
    \item The number of vertices is $5m$.
    \item An optimal IP solution picks both vertices of degree $3$ in each diamond and any two vertices in each triangle, so the optimal IP value is $3m$. This also implies that, at the root node, the maximal set of integer variables $I$ is the set of vertices that make up the diamonds.
    \item An optimal LP solution sets all vertices to have value $\frac{1}{2}$, so the optimal LP value is $\frac{5}{2}m$. 
    \item The additive integrality gap of the instance is then $\frac{1}{2}m$. 
\end{itemize}

\begin{figure}[ht]
\centering
\includegraphics[width=0.35\textwidth]{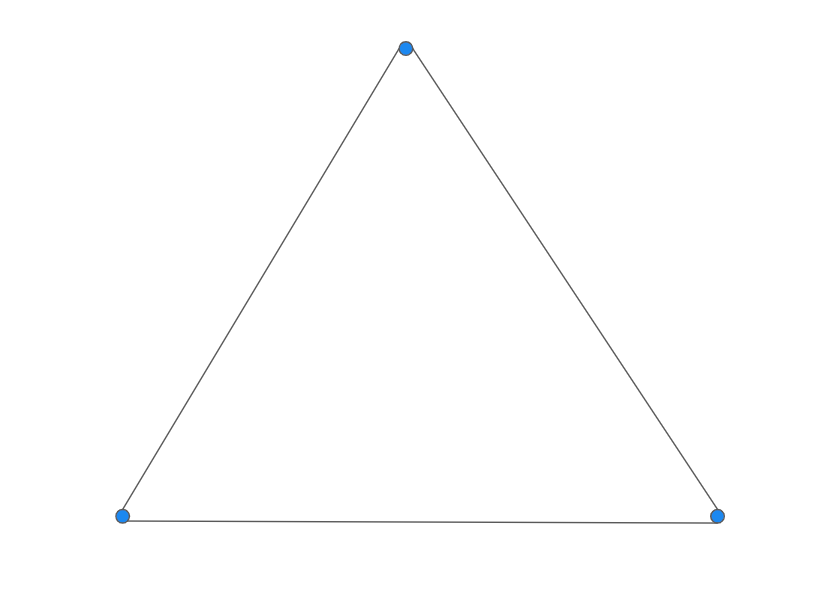}
\hspace{0.08\textwidth}
\includegraphics[width=0.4\textwidth]{diamond.png}
\caption{We refer to the component on the \textbf{left} as a \textit{triangle} and the component on the \textbf{right} as a \textit{diamond}. We label the vertices of the diamond for ease of following the arguments of this section. }
\label{fig:graphs}
\end{figure}

\paragraph{Branching on a diamond.} Consider any diamond, and suppose all of its vertices are set to $\frac{1}{2}$ in the LP optimal solution. Let $a$ be a vertex of degree $3$ and $b$ be a vertex of degree $2$; see Figure \ref{fig:graphs}. First, we consider the case of branching on $a$. Setting $x_a = 0$, the only feasible solution is to set the three other vertices in the diamond to $1$. This increases the objective value by $1$. Setting $x_a = 1$, the only optimal solution sets $x_c = 1$ and $x_b = x_d = 0$. This results in no objective value change. Therefore, $\text{score}_{P}(a) = 0$. Now, we consider the case of branching on $b$. Setting $x_b = 0$, any feasible solution sets $x_a = x_c = 1$ and the only optimal solution then sets $x_d = 0$. This results in no change to the objective value. Setting $x_v = 1$, the only optimal solution sets $x_a = x_c = x_d = \frac{1}{2}$. This results in an objective value increase of $\frac{1}{2}$. Therefore, $\text{score}_{P}(b) = 0$. 

\paragraph{Branching on a triangle.} Consider any triangle, and suppose all of its vertices are set to $\frac{1}{2}$ in the LP optimal solution. Let $u$ be any vertex of the triangle. Setting $x_u = 0$, the only optimal solution is to set the other two vertices to $1$. This results in an objective value increase of $\frac{1}{2}$. Setting $x_u = 1$, any optimal basic feasible solution sets one of the other two vertices to $1$ and the other to $0$. This results in an objective value increase of $\frac{1}{2}$. Therefore, $\text{score}_{P}(u) = \frac{1}{4}$. 

\paragraph{Strong branching.} Observe that, given an instance with at least one diamond with all of its vertices set to $\frac{1}{2}$ and at least one triangle with all of its vertices set to $\frac{1}{2}$, strong-branching will choose to branch on a vertex in such a triangle. Then, observe that each node on the $m$-th level of the strong-branching tree ''looks the same'': its LP optimal solution has an integer value for each vertex in every triangle, in particular, sets two vertices to $1$ and one vertex to $0$; and all other vertices (i.e. all vertices in the diamonds) are set to $\frac{1}{2}$. Observe that each of these nodes has LP value the \textit{same} as the optimal IP value. Since each node on the $m$-th level has an optimal solution that is integer, it follows from Lemma \ref{lem:branch_int} that strong-branching results in a tree of size at most $2^{m + 1} + \mathcal{O}(m)$.

\paragraph{A greedy rule.} Consider instead the \textit{greedy} variable selection rule: branch on the vertex with most fractional neighbors. Observe that, given an instance with at least one diamond with all of its vertices set to $\frac{1}{2}$ and at least one triangle with all of its vertices set to $\frac{1}{2}$, greedy will choose to branch on a vertex with degree $3$ in such a diamond (without loss of generality, we assume its vertex $a$). Now, consider the nodes on the $\frac{1}{2}m$-th level of the tree generated by greedy. Each node's optimal LP solution has an integer value for each vertex in every diamond, and all other vertices (i.e. all vertices in the triangles) are set to $\frac{1}{2}$. In particular, consider any specific node on this level and suppose that, on the path from the root to this node in the tree, $k$ vertices are set to $1$ and $\frac{1}{2}m - k$ vertices are set to $0$. Then, this node has objective value $\OPT(L(\I)) + (\frac{1}{2}m - k)$. Notice then that the sub-tree at this node will have to branch on $2k + 1$ triangles and the resulting $2^{2k + 1}$ leaves will have objective value $\OPT(L(\I)) + (\frac{1}{2}m - k) + \frac{1}{2}(2k + 1) = \OPT(L(\I)) + \frac{1}{2}m + \frac{1}{2} > \OPT(\I)$ (branching on any fewer triangles will not allow the sub-tree to be pruned). On the $\frac{1}{2}m$-th level, there are $\binom{m/2}{k}$ nodes with $k$ vertices set to $1$ on its path from the root. Then, the final number of nodes in the tree is more than
\begin{align*}
    \sum_{k = 1}^{m/2} \binom{m/2}{k} 2^{2k}.
\end{align*}
We can bound this quantity using the Binomial Theorem. In particular, we use the form
\begin{align}
    (x + y)^n  = \sum_{k = 0}^n \binom{n}{k} x^k y^{n - k}. \label{eq:binom}
\end{align}
Plugging $n = \frac{1}{2}, x = 4, y = 1$ into (\ref{eq:binom}), we have the bound 
\begin{align}
    \sum_{k = 0}^{m/2} \binom{m/2}{k} 4^k = 5^{m/2} \geq 2^{1.15 m}. \label{eq:final_bound}
\end{align}

%##################################################
%##################################################
%##################################################
%##################################################

\section{Proof of Proposition \ref{prop:bdg_ub} and Corollary \ref{cor:bdg_lb}: A bad example}\label{sec:bdg_ex}

\begingroup
\def\theproposition{\ref{prop:bdg_ub}}
\begin{proposition}
For any integer program $\I$ with $n$ binary variables, there is an equivalent integer program that uses an extended formulation of the feasible region of $\I$ with $2n$ binary variables, which we refer to as $BDG(\I)$, that has the following property: there exists a branch and bound tree $\mathcal{T}^*(BDG(\I))$ that solves the instance $BDG(\I)$ and $|\mathcal{T}^*(BDG(\I))| \leq 4n + 1$.%For any binary mixed-integer program $\I$, there is a pure binary integer program $BDG(\I)$ defined on $2n$ variables such that there exists a branch and bound tree $\mathcal{T}^*(BDG(\I))$ that solves the instance $BDG(\I)$ and $|\mathcal{T}^*(BDG(\I))| \leq 2n + 1$.
\end{proposition}
\addtocounter{proposition}{-1}
\endgroup

\begin{proof}
Consider the instance $\I$:
\begin{align*}
    \textup{max} &\quad c^{\top}x  \\
    \textup{s.t.} &\quad x \in P,\; x \in \{0, 1\}^n.  \tag{$\I$}
\end{align*}
where $P \subseteq [0,1]^n$ is a polytope. In~\cite{bodur2017cutting}, Bodur, Dash and Gunluk construct the  extended formulation $Q \subseteq [0,1]^{2n}$ of $P$ as follows. For every vertex $x$ of $P$, construct a vertex $(x, y)$ of $Q$ where
$$y_i = \begin{cases}
1 & \text{ if } x_i \in \{0,1\} \\
0 & \text{ if } x_i \in (0,1)
\end{cases}.$$
Define $Q$ to be the convex hull of these vertices, we call this the \emph{BDG extended formulation for $P$}. We construct the equivalent IP, $BDG(\I)$,  as follows:
\begin{align*}
    \textup{max} &\quad c^{\top}x  \\
    \textup{s.t.} &\quad (x, y) \in Q ,\; x \in \{0, 1\}^n,\; y \in \{0, 1\}^. \tag{$\textup{BDG}(\I)$}
\end{align*}

For any IP $\I$, there exists a branch-and-bound tree with at most $4n + 1$ nodes that solves $BDG(\I)$. However, this tree does not remove the current LP optimal fractional point when branching. See Figure \ref{fig:bdg}. 

\begin{figure}[h]
    \centering
    \includegraphics[width=12cm]{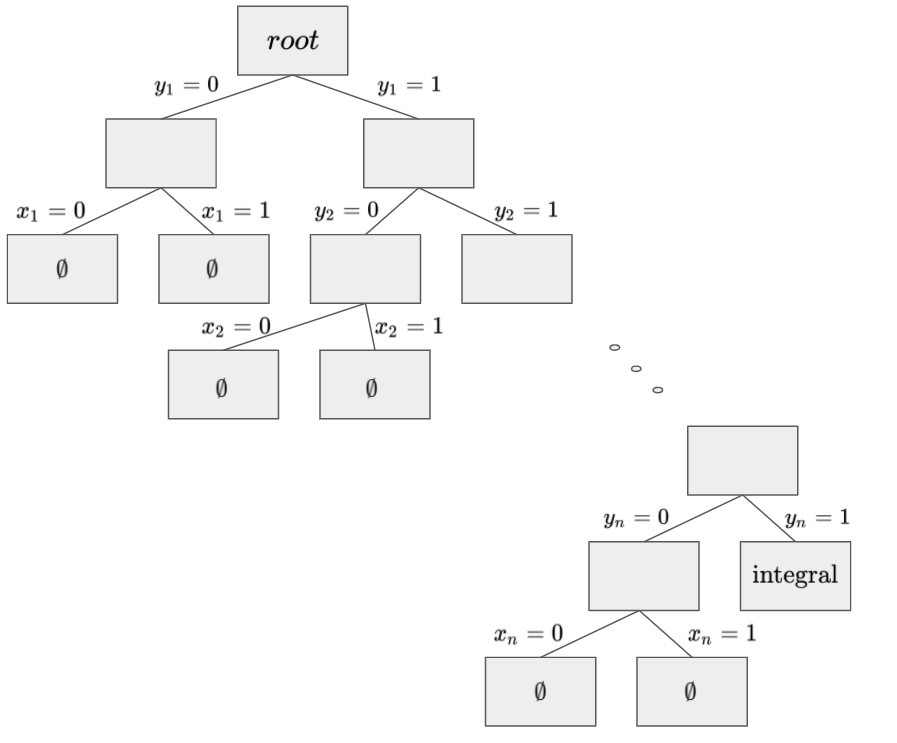}
    \caption{Branch-and-bound on BDG extended formulation}
    \label{fig:bdg}
\end{figure}

This follows since, by the definition of $Q$, all of its vertices that have $y_j = 0$ must have $x_j \in (0,1)$. Therefore, the branch that has $y_j = 0$ and $x_j = 0$ (and similarly $x_j = 1$) must be empty. Also, note that the branch that has $y_1 = 1, ..., y_n = 1$ must be integral. 
\end{proof}

\begingroup
\def\thecorollary{\ref{cor:bdg_lb}}
\begin{corollary}
There exists an instance $\I^*$ with $2n$ binary variables, such that the following holds: Let $\mathcal{T}(\I^*)$ be any tree that solves $\I^*$ satisfying the following property: if $x$ is the optimal solution to an internal node $N$ of $\mathcal{T}(\I^*)$, then the variable $j$ branched on at $N$ must be such that $x_j \in (0,1)$. Then, $|\mathcal{T}(\I^*)| \geq 2^{n +1} -1$.   In particular, if $\mathcal{T}_S(\I^*)$ is a branch-and-bound tree generated using strong-branching that solve $\I^*$, then $|\mathcal{T}_S(\I^*)| \geq 2^{n +1} - 1$.  On the other hand,  there exists a tree $\mathcal{T}^*$ that solves $\I^*$ such that $|\mathcal{T}^*(\I^*)| \leq 4n + 1$.
%There exists an instance $\I$ such that the following holds. Let $\mathcal{T}$ be any tree that solve $BDG(\I)$ satisfying the following property: if $x$ is the optimal solution to an internal node $N$ of $\mathcal{T}$, then the variable $j$ branched on at $N$ must be such that $x_j \in (0,1)$. Then, $|\mathcal{T}| \geq 2^n$.  In particular, if $\mathcal{T}_S(BDG(\I))$ is a branch-and-bound tree generated using strong-branching that solve $BDG(\I)$, then $|\mathcal{T}_S(BDG(\I))| \geq 2^n$.  
\end{corollary}
\addtocounter{corollary}{-1}
\endgroup

\begin{proof}
Note that the $y$ variables are not fractional in any vertex of $Q$. So when the tree of Figure \ref{fig:bdg} branches on a $y$ variable, it does not remove the current LP optimal fractional point (because it does not remove any vertex of $Q$). 

Now suppose we restrict ourselves to branching on a variable that does remove the current optimal point. Such a tree would only branch on $x$ variables (i.e. the original variables). Then, if $P$ is the $n$-dimensional cross polytope~\cite{dey2021lower}, and $Q$ is its BDG extended formulation, we know that branching on only the $x$ variables will require a tree of size at least $2^{n +1} -1$, as shown in \cite{dey2021lower}. 

The final statement follows from Proposition~\ref{prop:bdg_ub}.
%The first result follows. The result on strong-branching follows directly by considering the problem $\max_{x \in P \cap \{0,1\}^n} \ip{c}{x}$, where the components of $c$ corresponding to the $x$ variables are $1$ and those corresponding to the $y$ variables are $0$. 
\end{proof}

%%############################################################
%%############################################################
%%############################################################
%%############################################################

\section{Computing an optimal branch-and-bound tree}\label{sec:computing_opt}
\subsection{Proof of Observation~\ref{obs:assumption}.}\label{sec:opttreewelldefined}
\begingroup
\def\theobservation{\ref{obs:assumption}}
\begin{observation}
When using an LP solver which satisfies Assumption \ref{assum:int}, a branch-and-bound tree using the worst bound rule  never branches on a node whose optimal objective function value is equal to that of the IP solver. 
\end{observation}
\addtocounter{observation}{-1}
\endgroup

\begin{proof}
Consider a linear program at a node whose optimal function value is equal to that of the MILP optimal objective function value. There are two cases to consider. Case 1: if there exists an integral optimal solution to the LP relaxation at a given node, then the LP solver finds it and the node is pruned. Case 2: If not, then this node does not contain an integral solution with objective function value equal to that of the MILP optimal objective function value. This implies that there must exist another node whose feasible region contains an integer feasible solution with the same objective function value as that of the MILP optimal objective function value. In particular, this implies that there must exist a node whose optimal linear programming solution is such an  integer feasible solution. Since the LP solver will discover this solution before branching on the current node, we will never branch on the current node.
\end{proof}
%
%The above observation clearly makes the notion of ``optimal branch-and-bound tree" for a given instance well-defined: Restricting to any face of the zero-one cube (corresponding to branching-constraints), we precisely know if this node will be pruned or not. Thus, for a given set of branching-decision there is a deterministic number of nodes to solve the MILP, leading to an well-defined optimal number of nodes to solve an instance.

%In this section, we present a dynamic programming based algorithm to compute the optimal branch-and-bound tree, where optimality corresponds to the smallest possible tree size. In this paper, tree size is measured in terms of the number of times the tree branches. We then present implementation enhancements which decrease the running time and memory utilization of the algorithm.
 
\subsection{The dynamic programming algorithm to compute optimal branch-and-bound tree.} \label{sec:dpalgo}
We will now present the algorithm to compute the size of an optimal branch-and-bound tree for a fixed IP $\max_{x \in P \cap \{0,1\}^n}\ip{c}{x}$ under Assumption~\ref{assum:int}. Let $\mathcal{F}$ denote the set of faces of $[0,1]^n$ and note that each face can be defined as a string in $\{ \star , 0,  1\}^n$. For example, $(0, \star  ,1)$ denotes the face $\{x \in [0, 1]^3 : x_1 = 0, x_3 = 1\}$. Thus, $|\mathcal{F}| = 3^n$. Also, $\mathcal{F}$, is in one-to-one correspondence with all the possible nodes in the branch-and-bound tree. Let $\OPTn(F)$ denote the size of the optimal branch-and-bound tree for the sub-problem restricted to $F$, i.e., $\max_{x \in F \cap P \cap \{0,1\}^n}\ip{c}{x}$.  

%Since the goal is to compute the optimal size of the branch and bound tree, the algorithm is based on the premise that WDB rule is used for node selection as it results in the smallest tree size for any branching strategy. 
Based on Assumption~\ref{assum:int} and Observation~\ref{obs:assumption},  with the WDB rule for node selection, a node in the branch-and-bound tree is pruned if and only if it is either infeasible, or the optimal objective function value of its LP relaxation is less than or equal to the optimal MILP optimal objective value. In Phase-1 of our algorithm (Algorithm \ref{dpalgo}), this fact is used to identify nodes that are pruned by infeasibility or by bound, thus, $\OPTn(F) = 0$ for corresponding faces.

Now, given a face $F$ that is not pruned in the branch-and-bound tree, and variable $x_j$ that is free in $F$, define $F_{j,0}, F_{j,1}$ to be the faces of $F$ that result from fixing $x_j$ to $0$ and $1$ respectively, i.e. $F_{j, 0} = \{x \in F : x_j = 0\}$. The fact that the optimal sub-tree at the node corresponding to $F$ branches on the variable that produces two child nodes having the smallest optimal sub-trees, leads to the following recurrence relation, 
$$\OPTn(F) = 1 + \min_{j \in J_F} \left\{\OPTn(F_{j, 0}) + \OPTn(F_{j,1}) \right\},$$ 
where $J_F$ denotes the set of variables that are free in $F$. We use this recurrence relation in the bottom-up computation of $\OPTn(F)$ for the remaining faces (i.e. faces where  $\OPTn(F) \neq 0$) in $\mathcal{F}$ as Phase-2 of the algorithm. Thus, it can be inductively seen that the algorithm is correct. Additionally, the actual branch-and-bound tree can be found by storing $\arg \min_j \left\{\OPTn(F_{j, 0}) + \OPTn(F_{j,1})\right\}$ at every iteration. 

\begin{algorithm}
\caption{Computing Optimal Branch-and-bound Tree} \label{dpalgo}
\begin{algorithmic}[1]
\Statex \textbf{Phase-1:} Pruning by Infeasibility or Bound
\State Solve $\max_{x \in P \cap \{0,1\}^n}\ip{c}{x}$; let $x^*$ be the solution
\State Initialise: $\mathcal{S} \leftarrow \mathcal{F}$
\For{$F$ in $\mathcal{S}$}
    \State Solve $\max_{x \in F \cap P} \ip{c}{x}$; let $x_F^*$ be the optimal solution ($x_F^* = \emptyset$ if LP is infeasible)
    \If {$x_F^* = \emptyset$ \textbf{or} $\ip{c}{x_F^*} \le \ip{c}{x^*}$ }
        \State $\OPTn(F) \leftarrow 0$
        \State $\mathcal{S} \leftarrow \mathcal{S} \setminus \{F\}$
    \EndIf
\EndFor
\Statex
\Statex
\textbf{Phase-2:} Recursive bottom-up computation
\State Sort $\mathcal{S}$ in order of increasing dimension
\For{$F$ in $\mathcal{S}$}
    \State $\OPTn(F) \leftarrow 1 + \min_j (\OPTn(F_{j, 0}) + \OPTn(F_{j,1}))$
\EndFor
\State  \textbf{return} $\OPTn([0,1]^n)$
\end{algorithmic}
\end{algorithm}

\textcolor{black}{Notice it takes $2^{O(n)}$ time to execute line 1 and $\textup{poly}( data)$ time to execute line 3 of Algorithm \ref{dpalgo} for a particular face, where $\textup{poly}(data)$ is the running time for solving an LP.  Therefore, Phase I takes at most $2^{O(n)} + \textup{poly}(data) \cdot 3^n$ time.
Also notice Phase-2 takes $n \cdot 3^n$ time; this is because line 11 of Algorithm \ref{dpalgo} takes at most $n$ comparisons. So, in total Phase-1 and Phase-2 take $ 2^{O(n)} + (\textup{poly}(data) + n) \cdot 3^n = \textup{poly}(data)\cdot3^{O(n)}$ time. }

%\begingroup
%\def\thetheorem{\ref{thm:DP}}
%\begin{theorem} 
%Under Assumption~\ref{assum:int}, there exists an algorithm with running time $\textup{poly}(data)\cdot3^{O(n)}$ time to compute an optimal branch-and-bound tree for any binary MILP instance $\I$ defined on $n$ binary variables. 
%\end{theorem}
%\addtocounter{theorem}{-1}
%\endgroup

\subsection{Implementation enhancements}
\paragraph{Cascading 0-nodes} Observe that for $F_1, F_2 \in \mathcal{F}$ such that $F_2 \subset F_1$, if $\OPTn(F_1) = 0$, then $\OPTn(F_2) = 0$ as well. Thus,  at the start of Phase-1, faces in $\mathcal{F}$ are arranged in decreasing order of their dimension and upon finding a face, $F_1$, that is infeasible or pruned by bound, all other faces that are contained in $F_1$ can be removed from $\mathcal{S}$.

\paragraph{Parallelization} Phase-1 of the proposed algorithm can be directly parallelized, as each face can be independently solved. However, to benefit from Cascading 0-nodes, we group faces that have the same values from $\{0, 1, \star\}$ for the first $n_1$ dimensions and run Phase-1 for the resulting $3^{n_1}$ groups independently. Note that this also reduces peak memory usage as all $3^n$ faces in $\mathcal{F}$ are not required to be generated at the start of the algorithm and the remaining $n - n_1$ dimensions are generated on the fly for each group. In our experiments, we set $n_1 = \lfloor \frac{n}{2} \rfloor$. In addition, this manner of grouping also enables the LP solver to benefit from smaller incremental changes between consecutive faces. 

\section{Computational Experiments} \label{sec:experiments}
We evaluate the following branching strategies on problems mentioned in Section \ref{sec: Data generation} by comparing the number of branching to the optimal number of branching computed using the  dynamic programming based algorithm presented in Section \ref{sec:dpalgo}. Recall the definition of $\Delta^+_j, \Delta^-_j$ from Section \ref{sec:intro}. 
\begin{itemize}
\item \textbf{SB-L}: Strong branching where branching candidates are compared using a convex combination of the maximum and minimum. The scoring function used is $$ \text{score}_L(j) = \frac{5}{6} \min(\Delta^+_j, \Delta^-_j) + \frac{1}{6} \max(\Delta^+_j, \Delta^-_j).$$
%%\item \textbf{SB-T}: This is a variation of SB-N, motivated by the following observation. Consider the case when $\Delta \ge \langle c, x_{LP}\rangle - \langle c, x_{feas} \rangle$, where $x_{LP}$ is the solution of linear relaxation of the node and $x_{feas}$ is the best integer feasible solution found yet. Then the corresponding new node will be pruned by bound and the actual improvement is inconsequential. Let $\theta = \min(\Delta, \langle c, x_{LP} - x_{feas}\rangle)$, then the branching candidates are evaluated using, $\text{score}_C(\theta(\Delta^+), \theta(\Delta^-))$
\item \textbf{SB-P}: Strong branching where the product of improvements is used to score branching candidates. Thus, the scoring function is, $$\text{score}_P(j) = \max(\Delta^+_j, \epsilon) \cdot \max(\Delta^-_j, \epsilon),$$ where $\epsilon > 0$ is small. For the computations presented in this section, $\epsilon$ is chosen to be $10^{-4}.$
\item \textbf{Most-Inf}: Most infeasible branching where the variable with fractional value closest to 0.5 is selected for branching
\item \textbf{Rand}: Branching variable is randomly chosen from variables with fractional values.
%\item \textbf{Random-A}: Branching variable is randomly chosen from all free variables in the face. 
\end{itemize}
\subsection{Instance Generation} \label{sec: Data generation}
In this section, we discuss the problems considered for computational experiments and the details of generating randomized instances.

\subsubsection{General Packing and Covering IPs}
We consider packing problems, covering problems and general problems with multiple covering and packing inequalities. Let $I_p$ and $I_c$ represent the set of indices corresponding to packing and covering constraints respectively. The general formulation is as follows,
\begin{eqnarray*}
&\textup{max}& {c^{\top}}{x} \\
&\textup{s.t.} &a_i^{\top} x \leq b_i  \ \forall \ i \in I_p \\
&& a_i^{\top} x \geq b_i \ \forall \ i \in I_c \\
&& x \in \{0, 1\}^n.
\end{eqnarray*}
For the experiments in this paper, $n =  20$  is considered. The constraint matrix is generated randomly while incorporating sparsity. Each element $a_{ij}$ is $0$ with probability $p = 0.25$. Otherwise, a random integer selected uniformly from the set $\{1, 2, \dots, 200\}$ otherwise. The capacity parameter $b_i$ is $50\%$ of the sum of weights of the constraint, rounded down to integer value. The objective function is dependent on the number of packing and covering constraints as follows, 
\begin{itemize}
\item $P5$ is a purely packing type problem with $5$ constraints ($I_p = 5, I_c = 0$). We thus choose a non negative vector for the objective function, and each component $c_i$, is independently selected from the set $\{1, 2, \dots, 200\}$ with uniform probability.
\item $C5$ is a purely covering type problem with $5$ constraints ($I_p = 0, I_c = 5$). Each component of the objective, $c_i$, is independently selected from the set $\{-200, \dots, -2, -1\}$ with uniform probability.
\item$ G22$ is a general MILP with $2$ packing-type and $2$ covering-type constraints ($I_p~=~I_c~=~2$). Each component of the objective, $c_i$, is independently selected from the set $\{-100, \hdots, 100\}$ with uniform probability.
\end{itemize}

\subsubsection{Lot-sizing Problem and Variants}
We consider the classical lot-sizing problem of determining production volumes while minimizing production cost, fixed cost of production and inventory holding cost across the planning horizon. The MILP model for the lot-sizing problem with $n$ time periods is as follows,
\begin{eqnarray}
& \textup{min} & \sum_{i=1}^{n} \left( c_i \, x_i + f_i \, y_i \right ) +   \sum_{i=1}^{n -1} h_i \, s_i  \label{lotsize:obj} \\
&\textup{s.t.}& {x_1}{ = s_1 + d_1}{} \label{lotsize:flow0} \\
&& s_{i-1} + x_{i} = d_{i} + s_{i}  \ \forall \ i \in \{2, \dots, n\}  \label{lotsize:flow} \\
&& s_{n-1} + x_n = d_n \label{lotsize: flown}\\
&& x_i \le \left(\sum_{j=i}^n d_j \right) \, y_i \ \forall \ i \in \{1, \hdots, n\} \label{lotsize: x-y} \\
&& x \in  \mathbb{R}^n_+, \ s \in  \mathbb{R}^{n -1}_+, \  y \in  \{0, 1\}^n.
\end{eqnarray}
where variable $x_i$ is the quantity produced in period $i$ and $y_i$ is a binary variable with value $1$  if production occurs in period $i$ and $0$ otherwise. Lastly variable $s_i$ is the inventory at the end of period $i$. Unit cost of production for each period, $c_i$, is independently and uniformly sampled from $\{1, \hdots, 10\}$, fixed cost of production, $f_i$, from $\{200, \hdots, 400\}$ and unit inventory holding cost $h_i$ from $\{1, \hdots, 10\}$. Similarly, the demand for each time period, $d_i$ is independent and uniformly distributed in $\{0, \hdots, 100\}$. In our computational experiments, we consider problems with $n=17$.
\paragraph*{Capacitated Lot-sizing}
In the capacitated lot-sizing problem, the maximum quantity produced in every time period is constrained. Let parameter $u_i$ denote the upper-bound on the quantity that can be produced in period $i$. In our experiments, $u_i$ are sampled uniformly and independently from $\{150, \hdots, 250\}$. All other parameters are generated in the same was as the uncapacitated problem. Equation (\ref{lotsize: x-y}) is thus replaced with the constraint, 
\begin{equation*}
x_i \le u_i \, y_i, \qquad \forall i \in \{1, \hdots, n\}
\end{equation*}
\paragraph*{Big-bucket Lot-sizing}
The last variant of lot-sizing problem that we consider is the big-bucket lot-sizing problem where resources are shared amongst multiple products~\cite{quadt2008capacitated}. We do not consider unit cost of production here. On the other hand, set up time and processing time are considered to be constrained. 
%The problem is modelled as follows, 
%%\paragraph*{Parameters \\ \\}
The following are the parameters of the corresponding MILP model,\\ \\
\begin{tabular}{l l}
$P$ & Number of products, $\mathcal{P} = \{1, \hdots, P\}$ \\
$T$ & Number of time periods, $\mathcal{T} = \{1, \hdots, T\}$ \\
$f_i^p$ & Fixed cost of producing product $p$ in period $i$ \\
$h_i^p$ & Inventory holding cost of product $p$ in period $i$ \\
$t_i^{s, p}$ & Set up time of product $p$ in period $i$ \\
$t_i^{u, p}$ & Processing time per unit of product $p$ in period $i$ \\
$C_i$ & Time available in period $i$ \\
$z^p$ & Initial inventory of product $p$ at the beginning of planning horizon. \\
\end{tabular}
\newline \newline The variables used to model the problem are following,\\ \\
%%\paragraph*{Variables \\ \\}
\begin{tabular}{l l}
$x_i^p$ & Quantity product $p$ produced in period $i$ \\
$s_i^p$ & Quantity of product $p$ stored as inventory at the end of period $i$ \\
$y_i^p$ & Binary variable indicating if product $p$ was produced in period $i$ ($y_i^p = 1$ if $x_i^p > 0$)\\
\end{tabular}
\newline \newline The MILP model used for the big-bucket lot-sizing problem is described below.
%%\paragraph*{Model - Big bucket lot-sizing}
\begin{eqnarray*}
&\textup{min} &  \sum_{i \in \mathcal{T}} \sum_{p \in \mathcal{P}} \left(f^p_i \, y^p_i + h^p_i \, s^p_i \right) \label{bblotsize: obj} \\
& \textup{s.t.} & s^p_0 = z^p, \ p \in \mathcal{P}  \label{bblotsize: s_init}\\
&& s^p_T = 0, \ p \in \mathcal{P} \label{bblotsize: s_fin} \\
&& s^p_{i-1} + x^p_{i}  = d^p_{i} + s^p_{i},  i \in \mathcal{T},  p \in \mathcal{P} \label{bblotsize: flow} \\
&& x^p_i  \leq \left(\sum_{j=i}^n d^p_j \right) \, y^p_i, i \in \mathcal{T}, p \in \mathcal{P} \label{bblotsize: x-y} \\
&& \sum_{p \in \mathcal{P}} \left( t^{s,p}_i \, y^p_i + t^{u,p}_i \, x^p_i \right) \leq C_i,  i \in \mathcal{T} \label{bblotsize: cap} \\
&& x^p_i \in  \mathbb{R}_+, s^p_i \in  \mathbb{R}_+, y^p_i \in  \{0, 1\},\quad i \in \mathcal{T}, p \in \mathcal{P}.
\end{eqnarray*}
In our experiments, we consider problems with 9 time periods ($T=9$) and 2 products ($P=2$). Parameters corresponding to demand, fixed cost of production and inventory holding cost are generated as described in the context of previous variants. Set up time for a product in a each period, $t^{s, p}_i$ is independently sampled from $\{200, \hdots, 500\}$ with equal probability, unit processing time $t^{u, p}_i$ from $\{1, \hdots, 10\}$ and time limitation $C_i$ from $\{1000 P, \hdots, 2000 P\}$. Initial inventory for each product $z^p$ is similarly sampled from $\{0, \hdots, 200\}$.

\subsubsection{Minimum Vertex Cover}
The vertex cover problem on graphs $G = (V, E)$ concerns with identifying the smallest possible subset $V'$ of $V$, such that for every edge in $E$, at least one of its endpoints is included in $V'$. It is formulated as follows, 
\begin{eqnarray*}
&\textup{min}& \sum_{i \in V} x_i \\
&\textup{s.t.}& x_i + x_j \ge 1, \qquad \forall (i, j) \in E \\
&& x \in \{0,1\}^{|V|}
\end{eqnarray*}
We generated random graphs for the vertex cover problem using the \emph{Erd{\H o}s-R{\'e}nyi model}, i.e., the graph $G = (V, E)$  is constructed using 
two parameters; $N$ indicating the number of nodes and $p$ representing the probability with which each edge of the complete graph on $V$ is independently included in the set $E$. For our computational experiments, we consider graphs with $N=20$ and $p=0.75.$

%\subsubsection{Generalized Cropped Cube}
%Cornu\'{e}jols and Li~\cite{cornuejols2001rank} construct a family of polyhedrons $P_k \subseteq \mathbb{R}^n$ for $k\in \{1, \hdots, n\}$, with the property that $x$ is an extreme point of $P_k$ if and only if exactly $k$ of its components are $1/2$ and the remaining components are in $\{0, 1\}$. As discussed in Section~\ref{} ({\color{red} Connections to be added later}),  Bodur, Dash and G\"{u}nl\"{u}k~\cite{bodur2017cutting}, present an extended formulation $Q_k \subseteq \mathbb{R}^{2n}$ for $P_k$. We use the following MILP in our tests, with parametrized $Q_k$ as the feasible region and a random objective function, $c \in \{-100, \hdots, 0, \hdots, 100\}^n$,
%\begin{eqnarray*}
%&\textup{min} & \sum_{i = 1}^n c_i x_i \\
%&\textup{s.t.}& y_i \le p x_i, \qquad \forall i \in \{1, \hdots, n\} \\
%&& p x_i + y_i \le p, \qquad \forall i \in \{1, \hdots, n\} \\
%&& \sum_{i = 1}^n y_i = k \\
%&& x \in \{0,1\}^n, y \in \{0,1\}^n
%\end{eqnarray*} 

%({\color{red} We need to discuss this...}) 

\subsubsection{Chance Constraint Programming - Multiperiod Power Planning}

We consider the problem of expanding the electric power capacity~\cite{bertsimas1997introduction} of a state by constructing new coal and nuclear power plant to meet with the electricity demand of the state for a time horizon of $T$ periods. Once constructed, coal plants are operational for $T_c$ time periods and nuclear plants for $T_n$ time periods. Legal restrictions mandate that fraction of nuclear power should be at most $f$ of the total capacity. Capital cost incurred for the construction of coal and nuclear power plants operational from the beginning of time period $t$ are $c_t$ and $n_t$ respectively per megawatt of power capacity. The objective is to minimize the total capital cost of construction. Further, the demand is stochastic and defined on probability space $(\Omega, \mathcal{F}, \mathbb{P})$. Approximated by the sample approximation approach, $\Omega = \{\omega_1, \hdots, \omega_N\}$ is assumed to be a finite sample space. It is required that the probability of the event where the demand is not satisfied be at most $\epsilon$. The deterministic formulation of the problem as an MILP is as follows, 
\paragraph*{Parameters\\ \\}
\begin{tabular}{ l l }
$T$ & Number of time periods, $\mathcal{T} = \{1, \hdots, T\}$ \\
$N$ & Size of sample space $\Omega$,  $\mathcal{N} = \{1, \hdots, N\}$\\
$c_t$ & Capital cost per MW for coal plant operational from period $t$\\  
$n_t$ & Capital cost per MW for nuclear plant operational from period $t$ \\
$T_c$ & Lifespan of a coal power plant \\ 
$T_n$ & Lifespan of a nuclear power plant \\
$f$ & Upper bound on nuclear capacity as a fraction of total capacity \\
$e_t$ & Electric capacity from existing resources in period $t$ \\
$d^i_t$ & Electricity demand (in MW) in period t corresponding to outcome $\omega_i$\\
$p_i$ & Probability of outcome $\omega_i$\\
$\epsilon$ & Upper bound on the probability that the demand is satisfied
\end{tabular}

\paragraph*{Variables\\ \\}
\begin{tabular}{l l}
$x_t$ & Power capacity (in MW) of coal plants operational starting at period $t$ \\
$y_t$ & Power capacity (in MW) of coal plants operational starting at period $t$ \\
$u_t$ & Total coal power capacity (in MW) in period $t$ \\
$v_t$ & Total nuclear power capacity (in MW) in period $t$ \\
$z_i$ & Binary variable indicating if demand is not satisfied for outcome $\omega_i$
\end{tabular}

\paragraph*{Model - CCP Power Planning}
\begin{eqnarray}
&\textup{min}& \sum_{t=1}^{T} (c_t \, x_t + n_t \, y_t)  \label{ccp power: obj} \\
& \textup{s.t.} & u_t = \sum_{\textup{max}\{1, \, t-T_c+1\} }^t x_s, \quad t \in \mathcal{T} \label{ccp power: sum coal} \\
&& v_t = \sum_{\textup{max}\{1, \, t-T_n+1\}}^t y_s,\quad  t \in \mathcal{T} \label{ccp power: sum nucl} \\
&& (1-f) \, u_t - f \, v_t  \le f \, e_t, \quad t \in \mathcal{T} \label{ccp power: nuc lim}  \\
&&u_t + v_t \ge (d^i_t - e_t)(1-z_i), \quad t, i \in \mathcal{T} \times\mathcal{N} \label{ccp power: demand} \\
&& \sum_{i = 1}^n z_i \, p_i \le \epsilon \label{ccp power: epsilon} \\
&& x_t, y_t, u_t, v_t \in \mathbb{R}_+, \quad t \in \mathcal{T} \\
&& z_i  \in \{0, 1\}, \quad i \in \mathcal{N}.
\end{eqnarray}
The objective function (\ref{ccp power: obj}) minimizes total capital expenditure of constructing power plants. Equations (\ref{ccp power: sum coal}) and (\ref{ccp power: sum nucl}) compute total coal and nuclear power capacity for a given time period from active power plants based on their lifespan. Equation (\ref{ccp power: nuc lim}) enforces the regulatory limit on nuclear capacity is satisfied. Equations (\ref{ccp power: demand}) and (\ref{ccp power: epsilon}) ensure that the outcomes for which the demand is not satisfied has probability at most $\epsilon$.

For the experiments in Section \ref{sec:experiments}, we generate instances with $T = 30$ and $N = 20$ which corresponds to 30 time periods and 20 outcomes in sample space. Parameters $d^i_t$ are independent random integers uniformly distributed 
in $\{300, \hdots, 700\}$. Similarly, $c_t$ are uniformly distributed in $\{100, \hdots, 300\}$ and $n_t$ in $\{100, \hdots, 200\}$. Electric capacity from existing resources for the first period, $e_1$ is a random integer in $\{100, \hdots, 500 \}$. Capacity from existing resources is then modelled to decline by a factor of $r$ in every subsequent period where $r$ is uniformly distributed in $[0.7, 1)$, ie. $e_i = e_1\,r^{i-1}$. Lifespan of coal and nuclear power plants are $15$ and $10$ periods respectively. Nuclear capacity is constrained to be at most 20\% of total capacity. All outcomes in $\Omega$ are equally probable with $p_i = 0.05$ and demand satisfiability can be violated with a probability of at most $0.2$.

\subsubsection{Chance Constraint Programming - Portfolio Optimization}
We consider the probabilistically-constrained portfolio optimization problem for $n$ asset types, approximated by the sample approximation approach~\cite{pagnoncelli2009computational}, where the constraint on overall return may be violated for at most $k$ out of the $m$ samples. The MILP formulation of this problem is as follows:
\begin{eqnarray*}
&\textup{min} &\sum_{i = 1}^n x_i \\
&\textup{s.t.} & a_i^T x + r z_i\ge r, \quad \forall i = 1, \hdots, m \\
&& \sum_{i=1}^m z_i \le k \\
&& x \in \mathbb{R}^n_+, \, z \in \{0,1\}^m. 
\end{eqnarray*}
We sample scenarios from the distribution presented in~\cite{qiu2014covering}, which is shown to be computationally difficult to solve. Each component of the constraint matrix, $a_{ij}$ is independently sampled from a uniform distribution in $[0.8, 1.5]$ and $r$ is equal to $1.1$. For our experiments, we set $n=30$, $m=20$ and $k=4$.

\subsubsection{Stable Set Polytope on Bipartite Graph With Knapsack Constraint}
Stable set polytope corresponding to a bipartite graph is known to have a totally unimodular matrix and thus integral vertices. We consider the problem of solving a maximization problem on the stable set polytope of a bipartite graph where the optimal extreme point is cut off with a knapsack constraint with the same coefficients as the objective function. The details of model are explained below.  A bipartite graph $G = (N, E)$ is generated for a $n$ nodes and $m$ edges as follows. The partition of $N = N_1 \cup N_2$ is generated, by setting $N_1$ as a randomly selected subset of $\lfloor f n \rfloor$ nodes and $N_2$ as its complement, where $f$ is sampled from a uniform distribution over $[0.3, 0.5]$. From the $N_1 \times N_2$ possible edges, $m$ are then randomly selected to form set $E$. Lastly, each component $c_i$ of the objective function is a randomly selected integer from 1 to 50. In our experiments, we consider instances with $20$ nodes and $30$ edges.
Let $\delta^*$ be the objective function value of the corresponding maximum weight stable set problem, 
\begin{eqnarray}
&\textup{max}& \sum_{i=1}^n c_i \, x_i  \label{stable set ip}\\
&\textup{s.t.}& x_i + x_j \le 1,\quad \forall i, j \in E \label{stable set ip2}\\
&& x \in \{0, 1\}^m .\label{stable set ip3}
\end{eqnarray}
The following constraint is then added to (\ref{stable set ip})- (\ref{stable set ip3}),
\begin{equation}
\sum_{i=1}^n c_i \, x_i \le r\delta^*
\end{equation}
where $r$ is uniformly distributed in $[0.75, 0.9]$.

\subsection{Results}\label{sec:compres}
We present some preliminary results of various variable selection rules in comparison to optimal tree. $\textup{SB-L}$ stands for strong-branching with linear score, $\textup{SB-P}$ stands for strong-branching with product score, $\textup{Most infeasible}$ is selecting the variable with fractionality closest to $0.5$, and $\textup{Random}$ just selects a variable randomly from a list of fractional variables. 

Figure~\ref{fig:gmratio} shows a comparison of branching strategies based on the ratio of geometric mean of BB tree sizes to geometric mean of optimal tree sizes over all instances of the problem.

Figure~\ref{fig:performance1} shows performance profiles. The way to read the plots is the following: Consider the green curve for multi row packing problem. The point $(80\%, 0.3)$ means that in the case of using the most infeasible rule for multi row packing problem, $0.3\times 100 = 30$ instances out of 100, have branch-and-bound trees that have at most 80\% more nodes than the optimal branch-and-bound tree. The x-axis represents percentage differences in size of the BB tree in comparison to the optimal BB tree. The y-axis is the cumulative frequency of instances. 

Tables~1 through 12 present the number of nodes for all the instances we tested and all the variable selection rules. 

\subsubsection{Discussion On Findings}
It is clear from the Fig.~\ref{fig:gmratio}-\ref{fig:performance1} and Tables 1 through 12 that random always performs the worst and as expectated strong-branching is always the best. Furthermore, as seen in Fig.~\ref{fig:gmratio}, the geometric mean of tree size for strong-branching remains less than twice of the optimal tree for all problems considered in this study. The overall geometric mean of optimal tree size of instances across all problems is 42.65, whereas the same for SB-L, SB-P, most infeasible branching and random branching is 61.47, 61.09, 90.19 and 145.37 respectively. While the performance of two variants of strong-branching is comparable on all problems considered in this study, SB-P dominates over SB-L on 8 out of 10 problems, although by a small margin.

Strong branching is near optimal for the lot-sizing instances and performs exceptionally well on all other lot sizing variants as well. This is an interesting result considering that the tree sizes of the other two strategies can be significantly larger for almost all instances as seen in Fig~\ref{fig:scatter_lotsizing}. Thus, this indicates that very large branch-and-bound trees exist for the problem, but strong-branching is indeed effective at finding an effective branching strategy. This is in contrast to the problem of Stable Set on Bipartite Graphs with additional knapsack constraint, where all strategies have comparable performance as seen in Fig~\ref{fig:scatter_TUcutoff}. In comparison with the optimal tree, its performance on general packing and covering IPs is also relatively poor and is worst on Chance Constraint Programming Portfolio Optimization problem. The performance of random branching is also worst for CCP problems, thus making these an interesting class of problems for which better variable selection rules could be discovered.

%Finally, we caution the reader: Because of the exponential nature of the algorithm to compute the optimal branch-and-bound tree, computational experiments could be performed only on relatively smaller problem sizes with 17-20 binary variables, and it is not clear if how to extrapolate the results discovered to larger size instances.

\bibliographystyle{plain}
\bibliography{bib}	

\newpage

\appendix
\section{Graphs and tables}

\begin{figure}[ht]
     \begin{subfigure}[ht]{0.48\linewidth}
         \centering
         \includegraphics[width=\linewidth]{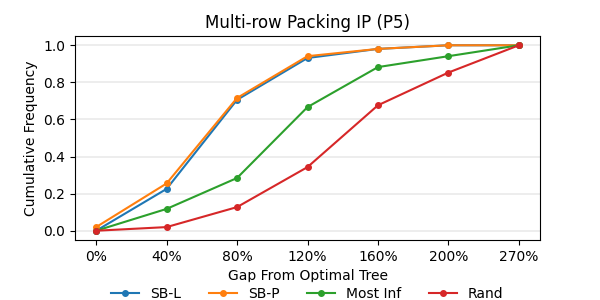}
         \caption{}
         
     \end{subfigure}
    %  \hfill
     \begin{subfigure}[ht]{0.48\linewidth}
         \centering
         \includegraphics[width=\linewidth]{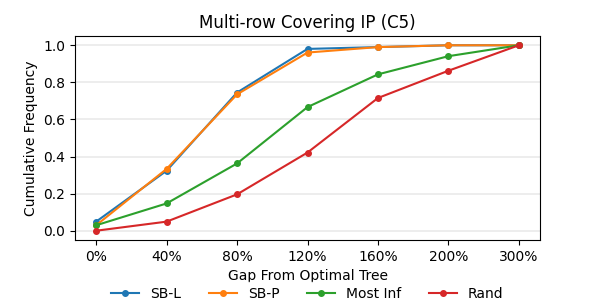}
         \caption{}
         
     \end{subfigure}
    %  \hfill
     \begin{subfigure}[ht]{0.48\linewidth}
         \centering
         \includegraphics[width=\linewidth]{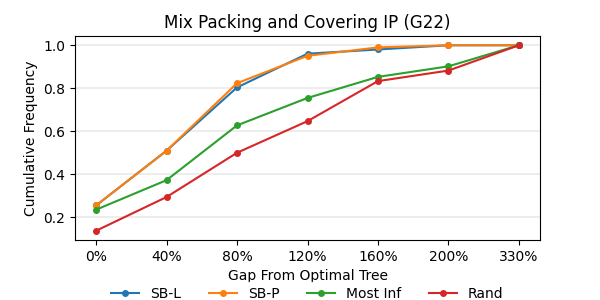}
         \caption{}
         
     \end{subfigure}
     \begin{subfigure}[ht]{0.48\linewidth}
         \centering
         \includegraphics[width=\linewidth]{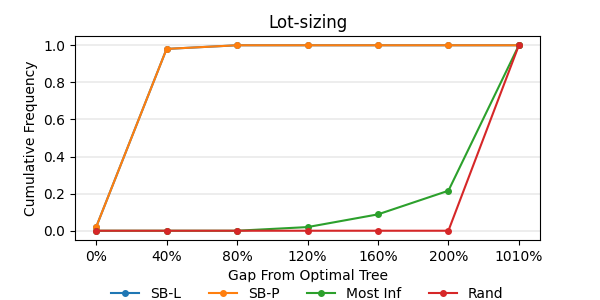}
         \caption{}
         
     \end{subfigure}
     \begin{subfigure}[ht]{0.48\linewidth}
         \centering
         \includegraphics[width=\linewidth]{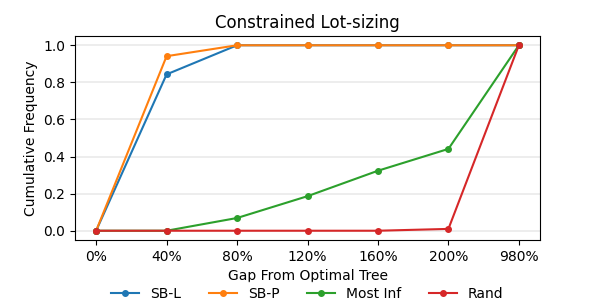}
         \caption{}
         
     \end{subfigure}
     \begin{subfigure}[ht]{0.48\linewidth}
         \centering
         \includegraphics[width=\linewidth]{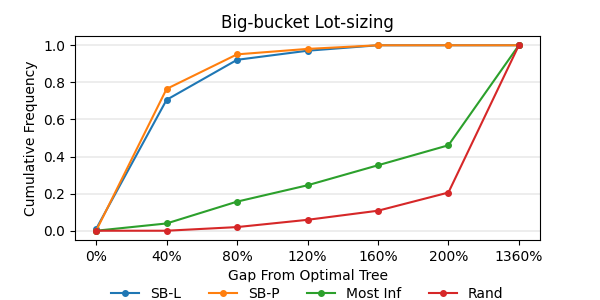}
         \caption{}
     \end{subfigure}
     
     \caption{Cumulative frequency of instances in terms of optimality gap for different branching strategies}
     
     \end{figure}
    
    \begin{figure}[ht]\ContinuedFloat
     
     \begin{subfigure}[ht]{0.48\linewidth}
         \centering
         \includegraphics[width=\linewidth]{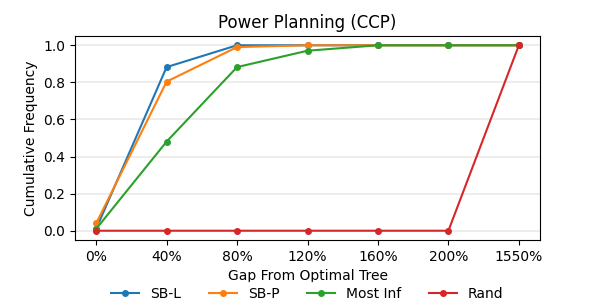}
         \caption{}
    
    \end{subfigure}
     \begin{subfigure}[ht]{0.48\linewidth}
         \centering
         \includegraphics[width=\linewidth]{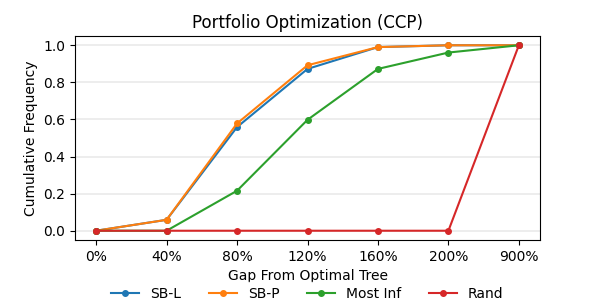}
         \caption{}
         
     \end{subfigure}
     \begin{subfigure}[ht]{0.48\linewidth}
         \centering
         \includegraphics[width=\linewidth]{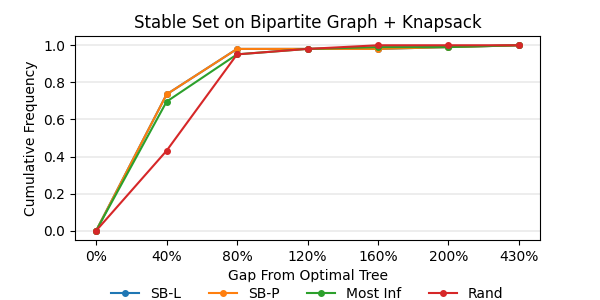}
         \caption{}
         
     \end{subfigure}
     \begin{subfigure}[ht]{0.48\linewidth}
         \centering
         \includegraphics[width=\linewidth]{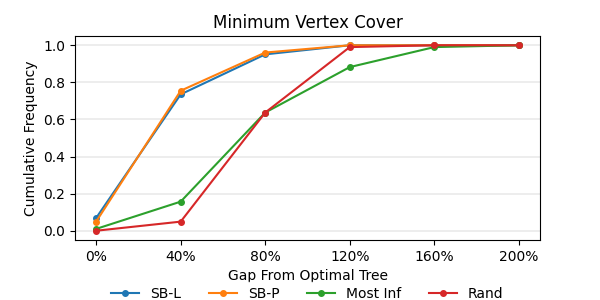}
         \caption{}
         
     \end{subfigure}
    %  \begin{subfigure}[h]{0.48\linewidth}
    %      \centering
    %      \includegraphics[width=0.8\linewidth]{Charts/vertex_cover_20_0.25.png}
    %      \caption{}
         
    %  \end{subfigure}
    %  \begin{subfigure}[h]{0.48\linewidth}
    %      \centering
    %      \includegraphics[width=0.8\linewidth]{Charts/vertex_cover_20_0.5.png}
    %      \caption{}
         
    %  \end{subfigure}
    %  \begin{subfigure}[h]{0.48\linewidth}
    %      \centering
    %      \includegraphics[width=0.8\linewidth]{Charts/vertex_cover_20_0.75.png}
    %      \caption{}
         
    %  \end{subfigure}
     
        \caption{Cumulative frequency of instances in terms of optimality gap for different branching strategies, cont'd.}
       \label{fig:performance1}  
\end{figure}

\begin{figure}
     \begin{subfigure}[h]{0.48\linewidth}
         \centering
         \includegraphics[width=0.9\linewidth]{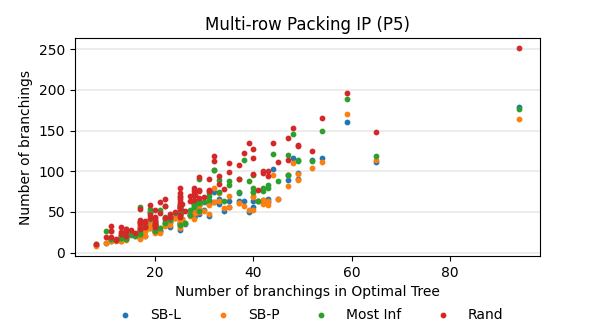}
         \caption{}
         
     \end{subfigure}
    %  \hfill
     \begin{subfigure}[h]{0.48\linewidth}
         \centering
         \includegraphics[width=0.9\linewidth]{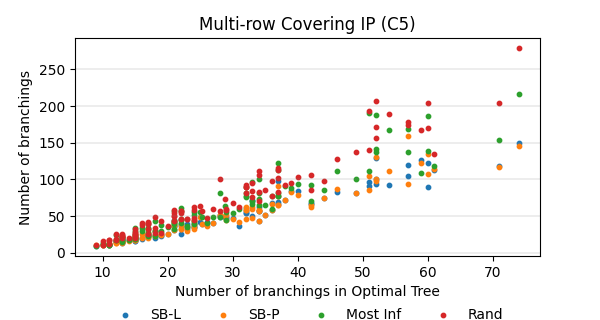}
         \caption{}
         
     \end{subfigure}
    %  \hfill
     \begin{subfigure}[h]{0.48\linewidth}
         \centering
         \includegraphics[width=0.9\linewidth]{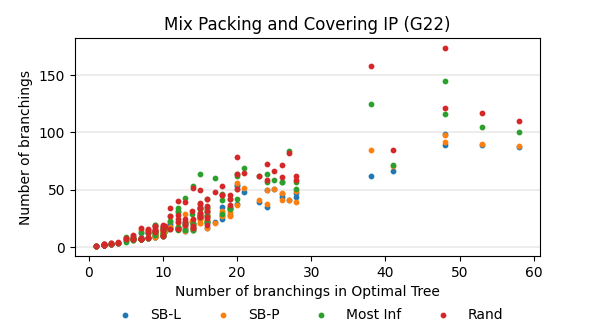}
         \caption{}
         
     \end{subfigure}
     \begin{subfigure}[h]{0.48\linewidth}
         \centering
         \includegraphics[width=0.9\linewidth]{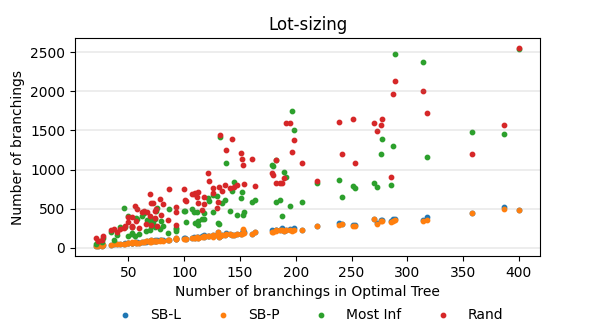}
         \caption{}\label{fig:scatter_lotsizing}
         
     \end{subfigure}
     \begin{subfigure}[h]{0.48\linewidth}
         \centering
         \includegraphics[width=0.9\linewidth]{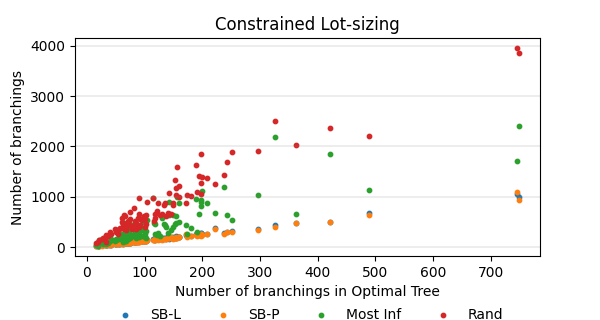}
         \caption{}\label{fig:scatter_constrained_LS}
         
     \end{subfigure}
     \begin{subfigure}[h]{0.48\linewidth}
         \centering
         \includegraphics[width=0.9\linewidth]{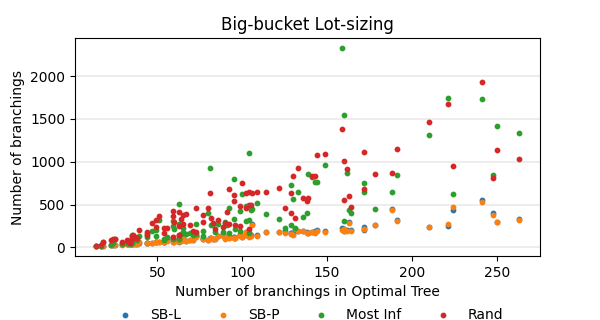}
         \caption{}\label{fig:scatter_BBLS}
         
     \end{subfigure}
     \begin{subfigure}[h]{0.48\linewidth}
         \centering
         \includegraphics[width=0.9\linewidth]{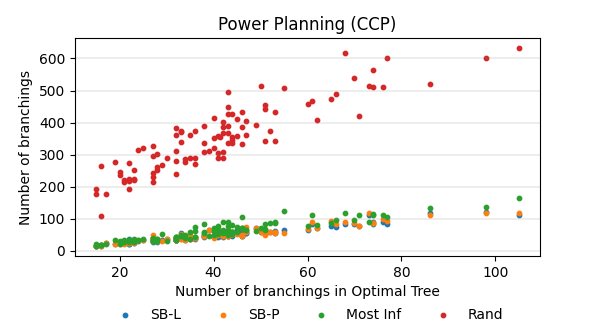}
         \caption{}
    
    \end{subfigure}
     \begin{subfigure}[h]{0.48\linewidth}
         \centering
         \includegraphics[width=0.9\linewidth]{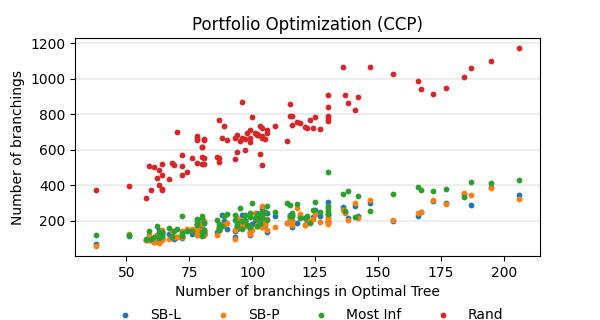}
         \caption{}
         
     \end{subfigure}
     \begin{subfigure}[h]{0.48\linewidth}
         \centering
         \includegraphics[width=0.9\linewidth]{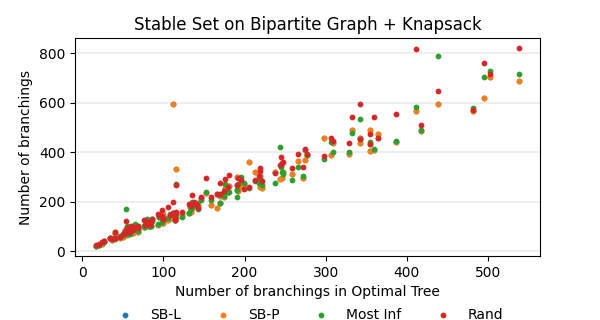}
         \caption{} \label{fig:scatter_TUcutoff}
         
     \end{subfigure}
     \begin{subfigure}[h]{0.48\linewidth}
         \centering
         \includegraphics[width=0.9\linewidth]{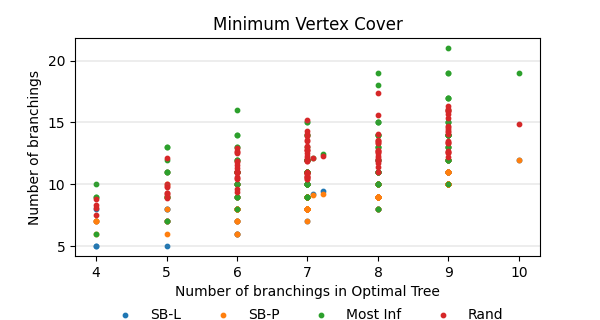}
         \caption{}
         
     \end{subfigure}\\
    
    \caption{Number of branching operations for all instances in comparison with the optimal tree}
    \label{fig:scatterplot_allinstances}  
\end{figure}

\small
\begin{table}\small
\begin{center}
\caption{Computational Results for Packing IP (P5) instances, $n=20, \; |I_p| = 5, \;|I_c| = 0, \; p=0.25$}
\begin{adjustwidth}{-0.5in}{-0.5in}
\begin{tabular}{crrrrrrrcrrrrr}
\cline{1-6} \cline{9-14}
\multicolumn{1}{p{1cm}}{\textbf{Index}} & \multicolumn{1}{p{1cm}}{\textbf{Opt Tree}} & \multicolumn{1}{p{1cm}}{\textbf{SB-L}} & \multicolumn{1}{p{1cm}}{\textbf{SB-P}} & \multicolumn{1}{p{1cm}}{\textbf{Most inf}} & \multicolumn{1}{p{1cm}}{\textbf{Rand}} &  &  & \multicolumn{1}{p{1cm}}{\textbf{Index}} & \multicolumn{1}{p{1cm}}{\textbf{Opt Tree}} & \multicolumn{1}{p{1cm}}{\textbf{SB-L}} & \multicolumn{1}{p{1cm}}{\textbf{SB-P}} & \multicolumn{1}{p{1cm}}{\textbf{Most inf}} & \multicolumn{1}{p{1cm}}{\textbf{Rand}} \\ \cline{1-6} \cline{9-14} 
1  & 37 & 73  & 61  & 90  & 91  &         &        & 52        & 27   & 47   & 46   & 52   & 63   \\
2  & 17 & 27  & 27  & 27  & 29  &         &        & 53        & 23   & 38   & 38   & 40   & 44   \\
3  & 28 & 44  & 56  & 60  & 80  &         &        & 54        & 13   & 19   & 19   & 18   & 22   \\
4  & 14 & 19  & 18  & 24  & 30  &         &        & 55        & 42   & 65   & 65   & 76   & 101  \\
5  & 49 & 98  & 89  & 113 & 132 &         &        & 56        & 47   & 95   & 96   & 120  & 141  \\
6  & 42 & 62  & 60  & 79  & 98  &         &        & 57        & 17   & 26   & 24   & 23   & 40   \\
7  & 49 & 90  & 97  & 114 & 131 &         &        & 58        & 45   & 66   & 66   & 88   & 112  \\
\end{tabular}
\end{adjustwidth}
\end{center}
\end{table}
\begin{table}\small
\begin{center}
\caption{Computational Results for Covering IP (C5) instances, $n=20, \; |I_p| = 0, \;|I_c| = 5, \; p=0.25$}
\begin{adjustwidth}{-0.5in}{-0.5in}
\begin{tabular}{crrrrrrrcrrrrr}
\cline{1-6} \cline{9-14}
\multicolumn{1}{p{1cm}}{\textbf{Index}} & \multicolumn{1}{p{1cm}}{\textbf{Opt Tree}} & \multicolumn{1}{p{1cm}}{\textbf{SB-L}} & \multicolumn{1}{p{1cm}}{\textbf{SB-P}} & \multicolumn{1}{p{1cm}}{\textbf{Most inf}} & \multicolumn{1}{p{1cm}}{\textbf{Rand}} &  &  & \multicolumn{1}{p{1cm}}{\textbf{Index}} & \multicolumn{1}{p{1cm}}{\textbf{Opt Tree}} & \multicolumn{1}{p{1cm}}{\textbf{SB-L}} & \multicolumn{1}{p{1cm}}{\textbf{SB-P}} & \multicolumn{1}{p{1cm}}{\textbf{Most inf}} & \multicolumn{1}{p{1cm}}{\textbf{Rand}} \\ \cline{1-6} \cline{9-14} 
1  & 21 & 37  & 31  & 32  & 43  &            &           & 52           & 21            & 31            & 38            & 39            & 50            \\
2  & 49 & 82  & 82  & 101 & 138 &            &           & 53           & 39            & 84            & 83            & 88            & 95            \\
3  & 11 & 11  & 13  & 11  & 18  &            &           & 54           & 16            & 20            & 25            & 30            & 38            \\
4  & 25 & 42  & 49  & 56  & 64  &            &           & 55           & 33            & 71            & 63            & 67            & 76            \\
5  & 36 & 68  & 67  & 77  & 98  &            &           & 56           & 37            & 69            & 65            & 78            & 83            \\
6  & 22 & 26  & 32  & 41  & 47  &            &           & 57           & 29            & 48            & 45            & 55            & 59            \\
7  & 28 & 49  & 50  & 82  & 101 &            &           & 58           & 42            & 67            & 65            & 71            & 85            \\
\end{tabular}
\end{adjustwidth}
\end{center}
\end{table}
\begin{table}\small
\begin{center}
\caption{Computational Results for Mix Packing and covering IP (G22) instances, $n=20, \; |I_p| = 2, \;|I_c| = 2, \; p=0.25$}
\begin{adjustwidth}{-0.5in}{-0.5in}
\begin{tabular}{crrrrrrrcrrrrr}
\cline{1-6} \cline{9-14}
\multicolumn{1}{p{1cm}}{\textbf{Index}} & \multicolumn{1}{p{1cm}}{\textbf{Opt Tree}} & \multicolumn{1}{p{1cm}}{\textbf{SB-L}} & \multicolumn{1}{p{1cm}}{\textbf{SB-P}} & \multicolumn{1}{p{1cm}}{\textbf{Most inf}} & \multicolumn{1}{p{1cm}}{\textbf{Rand}} &  &  & \multicolumn{1}{p{1cm}}{\textbf{Index}} & \multicolumn{1}{p{1cm}}{\textbf{Opt Tree}} & \multicolumn{1}{p{1cm}}{\textbf{SB-L}} & \multicolumn{1}{p{1cm}}{\textbf{SB-P}} & \multicolumn{1}{p{1cm}}{\textbf{Most inf}} & \multicolumn{1}{p{1cm}}{\textbf{Rand}} \\ \cline{1-6} \cline{9-14} 
1  & 10 & 17 & 15 & 18  & 20  &            &           & 52           & 12            & 16            & 17            & 18            & 22            \\
2  & 38 & 62 & 85 & 125 & 158 &            &           & 53           & 9             & 11            & 11            & 14            & 15            \\
3  & 48 & 99 & 98 & 145 & 173 &            &           & 54           & 20            & 38            & 37            & 42            & 51            \\
4  & 2  & 2  & 2  & 2   & 2   &            &           & 55           & 2             & 2             & 2             & 2             & 2             \\
5  & 18 & 27 & 30 & 29  & 45  &            &           & 56           & 15            & 26            & 25            & 26            & 30            \\
6  & 20 & 54 & 51 & 62  & 63  &            &           & 57           & 9             & 12            & 12            & 13            & 15            \\
7  & 8  & 15 & 15 & 14  & 16  &            &           & 58           & 15            & 35            & 33            & 64            & 50            \\
\end{tabular}
\end{adjustwidth}
\end{center}
\end{table}
\begin{table}\small
\begin{center}
\caption{Computational Results for Lot-sizing instances, $n=17$}
\begin{adjustwidth}{-0.5in}{-0.5in}
\begin{tabular}{crrrrrrrcrrrrr}
\cline{1-6} \cline{9-14}
\multicolumn{1}{p{1cm}}{\textbf{Index}} & \multicolumn{1}{p{1cm}}{\textbf{Opt Tree}} & \multicolumn{1}{p{1cm}}{\textbf{SB-L}} & \multicolumn{1}{p{1cm}}{\textbf{SB-P}} & \multicolumn{1}{p{1cm}}{\textbf{Most inf}} & \multicolumn{1}{p{1cm}}{\textbf{Rand}} &  &  & \multicolumn{1}{p{1cm}}{\textbf{Index}} & \multicolumn{1}{p{1cm}}{\textbf{Opt Tree}} & \multicolumn{1}{p{1cm}}{\textbf{SB-L}} & \multicolumn{1}{p{1cm}}{\textbf{SB-P}} & \multicolumn{1}{p{1cm}}{\textbf{Most inf}} & \multicolumn{1}{p{1cm}}{\textbf{Rand}} \\ \cline{1-6} \cline{9-14} 
1  & 79  & 94  & 94  & 339  & 416  &            &           & 52           & 37            & 44             & 44             & 102            & 239            \\
2  & 21  & 21  & 21  & 51   & 124  &            &           & 53           & 186           & 226            & 225            & 608            & 824            \\
3  & 76  & 102 & 87  & 504  & 278  &            &           & 54           & 137           & 191            & 174            & 1087           & 1249           \\
4  & 78  & 94  & 86  & 240  & 628  &            &           & 55           & 121           & 150            & 145            & 640            & 953            \\
5  & 358 & 448 & 442 & 1481 & 1204 &            &           & 56           & 251           & 291            & 279            & 784            & 1645           \\
6  & 72  & 85  & 85  & 312  & 573  &            &           & 57           & 74            & 94             & 89             & 364            & 384            \\
7  & 122 & 156 & 151 & 601  & 849  &            &           & 58           & 112           & 127            & 125            & 348            & 707            \\
\end{tabular}
\end{adjustwidth}
\end{center}
\end{table}
\begin{table}\small
\begin{center}
\caption{Computational Results for Capacitated Lot-sizing instances, $n=17$}
\begin{adjustwidth}{-0.5in}{-0.5in}
\begin{tabular}{crrrrrrrcrrrrr}
\cline{1-6} \cline{9-14}
\multicolumn{1}{p{1cm}}{\textbf{Index}} & \multicolumn{1}{p{1cm}}{\textbf{Opt Tree}} & \multicolumn{1}{p{1cm}}{\textbf{SB-L}} & \multicolumn{1}{p{1cm}}{\textbf{SB-P}} & \multicolumn{1}{p{1cm}}{\textbf{Most inf}} & \multicolumn{1}{p{1cm}}{\textbf{Rand}} &  &  & \multicolumn{1}{p{1cm}}{\textbf{Index}} & \multicolumn{1}{p{1cm}}{\textbf{Opt Tree}} & \multicolumn{1}{p{1cm}}{\textbf{SB-L}} & \multicolumn{1}{p{1cm}}{\textbf{SB-P}} & \multicolumn{1}{p{1cm}}{\textbf{Most inf}} & \multicolumn{1}{p{1cm}}{\textbf{Rand}} \\ \cline{1-6} \cline{9-14} 
1  & 60  & 85  & 71  & 178  & 572  &            &           & 52           & 33             & 37             & 38             & 106            & 246            \\
2  & 243 & 292 & 298 & 630  & 1694 &            &           & 53           & 96             & 120            & 124            & 333            & 563            \\
3  & 19  & 20  & 22  & 27   & 93   &            &           & 54           & 326            & 441            & 397            & 2189           & 2498           \\
4  & 198 & 233 & 225 & 901  & 1842 &            &           & 55           & 118            & 177            & 153            & 456            & 666            \\
5  & 114 & 158 & 148 & 983  & 969  &            &           & 56           & 21             & 28             & 27             & 134            & 143            \\
6  & 85  & 110 & 105 & 172  & 352  &            &           & 57           & 146            & 202            & 198            & 344            & 660            \\
7  & 90  & 144 & 109 & 207  & 404  &            &           & 58           & 90             & 120            & 120            & 223            & 663            \\
\end{tabular}
\end{adjustwidth}
\end{center}
\end{table}
\begin{table}\small
\begin{center}
\caption{Computational Results for Big-bucket Lot-sizing instances, $T=9, P=2$}
\begin{adjustwidth}{-0.5in}{-0.5in}
\begin{tabular}{crrrrrrrcrrrrr}
\cline{1-6} \cline{9-14}
\multicolumn{1}{p{1cm}}{\textbf{Index}} & \multicolumn{1}{p{1cm}}{\textbf{Opt Tree}} & \multicolumn{1}{p{1cm}}{\textbf{SB-L}} & \multicolumn{1}{p{1cm}}{\textbf{SB-P}} & \multicolumn{1}{p{1cm}}{\textbf{Most inf}} & \multicolumn{1}{p{1cm}}{\textbf{Rand}} &  &  & \multicolumn{1}{p{1cm}}{\textbf{Index}} & \multicolumn{1}{p{1cm}}{\textbf{Opt Tree}} & \multicolumn{1}{p{1cm}}{\textbf{SB-L}} & \multicolumn{1}{p{1cm}}{\textbf{SB-P}} & \multicolumn{1}{p{1cm}}{\textbf{Most inf}} & \multicolumn{1}{p{1cm}}{\textbf{Rand}} \\ \cline{1-6} \cline{9-14} 
1  & 172 & 202 & 210 & 642  & 683  &            &           & 52           & 160           & 191            & 191            & 1549           & 1004           \\
2  & 102 & 184 & 179 & 312  & 464  &            &           & 53           & 77            & 109            & 96             & 194            & 383            \\
3  & 80  & 92  & 93  & 402  & 460  &            &           & 54           & 32            & 40             & 40             & 58             & 84             \\
4  & 63  & 68  & 70  & 271  & 249  &            &           & 55           & 248           & 397            & 379            & 847            & 811            \\
5  & 24  & 29  & 27  & 58   & 93   &            &           & 56           & 29            & 36             & 33             & 45             & 64             \\
6  & 73  & 129 & 123 & 159  & 177  &            &           & 57           & 149           & 193            & 184            & 961            & 1092           \\
7  & 60  & 71  & 69  & 216  & 303  &            &           & 58           & 106           & 274            & 259            & 447            & 636            \\
\end{tabular}
\end{adjustwidth}
\end{center}
\end{table}
\begin{table}\small
\begin{center}
\caption{Computational Results for Multi-period power planning (CCP), $T=30, N=20$}
\begin{adjustwidth}{-0.5in}{-0.5in}
\begin{tabular}{crrrrrrrcrrrrr}
\cline{1-6} \cline{9-14}
\multicolumn{1}{p{1cm}}{\textbf{Index}} & \multicolumn{1}{p{1cm}}{\textbf{Opt Tree}} & \multicolumn{1}{p{1cm}}{\textbf{SB-L}} & \multicolumn{1}{p{1cm}}{\textbf{SB-P}} & \multicolumn{1}{p{1cm}}{\textbf{Most inf}} & \multicolumn{1}{p{1cm}}{\textbf{Rand}} &  &  & \multicolumn{1}{p{1cm}}{\textbf{Index}} & \multicolumn{1}{p{1cm}}{\textbf{Opt Tree}} & \multicolumn{1}{p{1cm}}{\textbf{SB-L}} & \multicolumn{1}{p{1cm}}{\textbf{SB-P}} & \multicolumn{1}{p{1cm}}{\textbf{Most inf}} & \multicolumn{1}{p{1cm}}{\textbf{Rand}} \\ \cline{1-6} \cline{9-14} 
1  & 29  & 34  & 32  & 54  & 269 &            &           & 52           & 28            & 31            & 33            & 31            & 263            \\
2  & 53  & 61  & 58  & 89  & 343 &            &           & 53           & 55            & 64            & 57            & 126           & 508            \\
3  & 73  & 111 & 118 & 90  & 515 &            &           & 54           & 46            & 48            & 50            & 72            & 386            \\
4  & 70  & 85  & 87  & 96  & 539 &            &           & 55           & 74            & 91            & 88            & 111           & 511            \\
5  & 42  & 58  & 59  & 66  & 290 &            &           & 56           & 62            & 71            & 71            & 82            & 407            \\
6  & 16  & 16  & 16  & 20  & 264 &            &           & 57           & 22            & 38            & 26            & 32            & 192            \\
7  & 41  & 51  & 67  & 78  & 307 &            &           & 58           & 46            & 48            & 46            & 106           & 434            \\
\end{tabular}
\end{adjustwidth}
\end{center}
\end{table}
\begin{table}\small
\begin{center}
\caption{Computational Results for Portfolio Optimization (CCP), $n= 30$, $m= 20$, $k= 4$}
\begin{adjustwidth}{-0.5in}{-0.5in}
\begin{tabular}{crrrrrrrcrrrrr}
\cline{1-6} \cline{9-14}
\multicolumn{1}{p{1cm}}{\textbf{Index}} & \multicolumn{1}{p{1cm}}{\textbf{Opt Tree}} & \multicolumn{1}{p{1cm}}{\textbf{SB-L}} & \multicolumn{1}{p{1cm}}{\textbf{SB-P}} & \multicolumn{1}{p{1cm}}{\textbf{Most inf}} & \multicolumn{1}{p{1cm}}{\textbf{Rand}} &  &  & \multicolumn{1}{p{1cm}}{\textbf{Index}} & \multicolumn{1}{p{1cm}}{\textbf{Opt Tree}} & \multicolumn{1}{p{1cm}}{\textbf{SB-L}} & \multicolumn{1}{p{1cm}}{\textbf{SB-P}} & \multicolumn{1}{p{1cm}}{\textbf{Most inf}} & \multicolumn{1}{p{1cm}}{\textbf{Rand}} \\ \cline{1-6} \cline{9-14} 
1  & 130 & 237 & 215 & 284 & 761  &            &           & 52           & 81             & 120            & 132            & 187            & 552            \\
2  & 122 & 201 & 210 & 228 & 722  &            &           & 53           & 156            & 199            & 204            & 351            & 1025           \\
3  & 63  & 149 & 101 & 132 & 488  &            &           & 54           & 137            & 248            & 243            & 255            & 910            \\
4  & 96  & 235 & 186 & 157 & 668  &            &           & 55           & 102            & 172            & 178            & 170            & 684            \\
5  & 63  & 89  & 78  & 106 & 400  &            &           & 56           & 116            & 233            & 203            & 244            & 740            \\
6  & 81  & 126 & 118 & 138 & 522  &            &           & 57           & 147            & 300            & 316            & 255            & 1066           \\
7  & 97  & 201 & 201 & 232 & 599  &            &           & 58           & 136            & 278            & 260            & 350            & 1067           \\
\end{tabular}
\end{adjustwidth}
\end{center}
\end{table}
\begin{table}\small
\begin{center}
\caption{Computational Results for Stable Set Polytope on Bipartite Graph with Knapsack Constraint}
\begin{adjustwidth}{-0.5in}{-0.5in}
\begin{tabular}{crrrrrrrcrrrrr}
\cline{1-6} \cline{9-14}
\multicolumn{1}{p{1cm}}{\textbf{Index}} & \multicolumn{1}{p{1cm}}{\textbf{Opt Tree}} & \multicolumn{1}{p{1cm}}{\textbf{SB-L}} & \multicolumn{1}{p{1cm}}{\textbf{SB-P}} & \multicolumn{1}{p{1cm}}{\textbf{Most inf}} & \multicolumn{1}{p{1cm}}{\textbf{Rand}} &  &  & \multicolumn{1}{p{1cm}}{\textbf{Index}} & \multicolumn{1}{p{1cm}}{\textbf{Opt Tree}} & \multicolumn{1}{p{1cm}}{\textbf{SB-L}} & \multicolumn{1}{p{1cm}}{\textbf{SB-P}} & \multicolumn{1}{p{1cm}}{\textbf{Most inf}} & \multicolumn{1}{p{1cm}}{\textbf{Rand}} \\ \cline{1-6} \cline{9-14} 
1  & 17  & 20  & 20  & 21  & 25.5  &         &        & 52        & 192  & 243   & 243   & 273   & 294.5 \\
2  & 481 & 567 & 567 & 579 & 569.5 &         &        & 53        & 134  & 159   & 159   & 174   & 188.6 \\
3  & 62  & 73  & 73  & 91  & 83.9  &         &        & 54        & 68   & 78    & 78    & 83    & 92.9  \\
4  & 213 & 319 & 319 & 282 & 286.9 &         &        & 55        & 93   & 104   & 104   & 109   & 151.7 \\
5  & 364 & 472 & 472 & 457 & 459.6 &         &        & 56        & 221  & 257   & 257   & 266   & 282.1 \\
6  & 85  & 102 & 102 & 101 & 117.7 &         &        & 57        & 219  & 304   & 304   & 293   & 322.6 \\
7  & 170 & 195 & 195 & 194 & 229.9 &         &        & 58        & 258  & 310   & 310   & 287   & 337.2 \\
\end{tabular}
\end{adjustwidth}
\end{center}
\end{table}
\begin{table}\small
\begin{center}
\caption{Computational Results for Vertex Cover instances with $|V|=20,\; p=0.75$}
\begin{adjustwidth}{-0.5in}{-0.5in}
\begin{tabular}{crrrrrrrcrrrrr}
\cline{1-6} \cline{9-14}
\multicolumn{1}{p{1cm}}{\textbf{Index}} & \multicolumn{1}{p{1cm}}{\textbf{Opt Tree}} & \multicolumn{1}{p{1cm}}{\textbf{SB-L}} & \multicolumn{1}{p{1cm}}{\textbf{SB-P}} & \multicolumn{1}{p{1cm}}{\textbf{Most inf}} & \multicolumn{1}{p{1cm}}{\textbf{Rand}} &  &  & \multicolumn{1}{p{1cm}}{\textbf{Index}} & \multicolumn{1}{p{1cm}}{\textbf{Opt Tree}} & \multicolumn{1}{p{1cm}}{\textbf{SB-L}} & \multicolumn{1}{p{1cm}}{\textbf{SB-P}} & \multicolumn{1}{p{1cm}}{\textbf{Most inf}} & \multicolumn{1}{p{1cm}}{\textbf{Rand}} \\ \cline{1-6} \cline{9-14}
1  & 8 & 8  & 8  & 11 & 13 &            &           & 52           & 9            & 11           & 10           & 17            & 16            \\
2  & 7 & 8  & 8  & 14 & 14 &            &           & 53           & 5            & 7            & 9            & 12            & 9             \\
3  & 9 & 10 & 11 & 14 & 15 &            &           & 54           & 8            & 11           & 9            & 8             & 12            \\
4  & 7 & 9  & 9  & 12 & 12 &            &           & 55           & 8            & 9            & 9            & 13            & 12            \\
5  & 9 & 14 & 10 & 15 & 13 &            &           & 56           & 8            & 10           & 9            & 18            & 17            \\
6  & 7 & 9  & 10 & 13 & 15 &            &           & 57           & 9            & 10           & 11           & 10            & 13            \\
7  & 7 & 9  & 8  & 9  & 11 &            &           & 58           & 8            & 11           & 10           & 13            & 14            \\
\end{tabular}
\end{adjustwidth}
\end{center}
\end{table}
%\include{tables/table_vertex_cover_20,50a}
%\include{tables/table_vertex_cover_20,25a}
%\include{tables/table_gen_cropcube1}
%
%

%############################################################
%############################################################
%############################################################
%############################################################

\end{document}